\documentclass[10pt]{amsart}
\usepackage{amsmath}
\usepackage{amsxtra}
\usepackage{amscd}
\usepackage{amsthm}
\usepackage{amsfonts}
\usepackage{amssymb}
\usepackage{eucal}
\usepackage[all]{xy}
\usepackage{graphicx}
\usepackage[usenames]{color}

\newtheorem{cor}[subsubsection]{Corollary}
\newtheorem{lem}[subsubsection]{Lemma}
\newtheorem{prop}[subsubsection]{Proposition}

\newtheorem{conj}[subsubsection]{Conjecture}
\newtheorem{thm}[subsubsection]{Theorem}

\theoremstyle{remark}
\newtheorem{rem}[subsubsection]{Remark}

\theoremstyle{definition}

\theoremstyle{remark}

\newcommand{\thmref}[1]{Theorem~\ref{#1}}

\newcommand{\secref}[1]{Sect.~\ref{#1}}
\newcommand{\lemref}[1]{Lemma~\ref{#1}}
\newcommand{\propref}[1]{Proposition~\ref{#1}}
\newcommand{\corref}[1]{Corollary~\ref{#1}}
\newcommand{\conjref}[1]{Conjecture~\ref{#1}}

\numberwithin{equation}{section}

\newcommand{\nc}{\newcommand}
\nc{\renc}{\renewcommand}
\nc{\ssec}{\subsection}
\nc{\sssec}{\subsubsection}
\nc{\on}{\operatorname}

\nc{\ips}{{\iota_P^{(S)}}}
\nc{\ipms}{{\iota_{P^-}^{(S)}}}
\nc{\sfpps}{{\sfp_P^{(S)}}}
\nc{\sfppms}{{\sfp_{P^-}^{(S)}}}

\nc\ol{\overline}
\nc\wt{\widetilde}
\nc\tboxtimes{\wt{\boxtimes}}
\nc\tstar{\wt{\star}}
\nc{\alp}{\alpha}

\nc{\ZZ}{{\mathbb Z}}
\nc{\NN}{{\mathbb N}}
\nc{\BF}{{\mathbb F}}
\nc{\OO}{{\mathbb O}}
\renc{\SS}{{\mathbb S}}
\nc{\DD}{{\mathbb D}}
\nc{\GG}{{\mathbb G}}

\nc{\Fq}{{\mathbb F}_q}
\nc{\Fqb}{\ol{\mathbb F}_q}
\nc{\Ql}{{\mathbb Q}_\ell}
\nc{\Qlb}{{\ol{\mathbb Q}_\ell}}
\nc{\id}{\text{id}}
\nc\X{\mathcal X}

\nc{\red}{\on{red}}
\nc{\Ho}{\on{Ho}}
\nc{\Hom}{\on{Hom}}
\nc{\coef}{\on{coef}}
\nc{\Lie}{\on{Lie}}
\nc{\Diff}{\on{Diff}}
\nc{\Loc}{\on{Loc}}
\nc{\Pic}{\on{Pic}}
\nc{\Bun}{\on{Bun}}
\nc{\IC}{\on{IC}}
\nc{\Aut}{\on{Aut}}
\nc{\rk}{\on{rk}}
\nc{\Sh}{\on{Sh}}
\nc{\Perv}{\on{Perv}}
\nc{\pos}{{\on{pos}}}
\nc{\Conv}{\on{Conv}}
\nc{\Sph}{\on{Sph}}
\nc{\Sym}{\on{Sym}}
\nc{\BunBb}{\overline{\Bun}_B}
\nc{\BunNb}{\overline{\Bun}_N}
\nc{\BunTb}{\overline{\Bun}_T}
\nc{\BunBbm}{\overline{\Bun}_{B^-}}
\nc{\BunBbel}{\overline{\Bun}_{B,el}}
\nc{\BunBbmel}{\overline{\Bun}_{B^-,el}}
\nc{\Buno}{\overset{o}{\Bun}}
\nc{\BunPb}{{\overline{\Bun}_P}}
\nc{\BunBM}{\Bun_{B(M)}}
\nc{\BunBMb}{\overline{\Bun}_{B(M)}}
\nc{\BunPbw}{{\widetilde{\Bun}_P}}
\nc{\BunBP}{\widetilde{\Bun}_{B,P}}
\nc{\GUb}{\overline{G/U}}
\nc{\GUPb}{\overline{G/U(P)}}

\nc{\Hhom}{\underline{\on{Hom}}}
\nc\syminfty{\on{Sym}^{\infty}}
\nc\lal{\ol{\lambda}}
\nc\xl{\ol{x}}
\nc\thl{\ol{\theta}}
\nc\nul{\ol{\nu}}
\nc\mul{\ol{\mu}}
\nc\Sum\Sigma
\nc{\oX}{\overset{\circ}{X}{}}
\nc{\wtoX}{\wt{\overset{\circ}{X}}{}}
\nc{\hl}{\overset{\leftarrow}h{}}
\nc{\hr}{\overset{\rightarrow}h{}}
\nc{\M}{{\mathcal M}}
\nc{\N}{{\mathcal N}}
\nc{\F}{{\mathcal F}}
\nc{\D}{{\mathcal D}}
\nc{\Q}{{\mathcal Q}}
\nc{\Y}{{\mathcal Y}}
\nc{\G}{{\mathcal G}}
\nc{\E}{{\mathcal E}}
\nc{\CalC}{{\mathcal C}}
\nc\Dh{\widehat{\D}}

\nc{\C}{{\mathcal C}}
\nc{\K}{{\mathcal K}}
\renewcommand{\H}{{\mathcal H}}

\nc{\T}{{\mathcal T}}
\nc{\V}{{\mathcal V}}
\renc{\P}{{\mathcal P}}
\nc{\A}{{\mathcal A}}
\nc{\B}{{\mathcal B}}
\nc{\U}{{\mathcal U}}

\nc{\Gr}{{\on{Gr}}}

\nc{\frn}{{\check{\mathfrak u}(P)}}

\nc{\fC}{\mathfrak C}
\nc{\p}{\mathfrak p}
\nc{\q}{\mathfrak q}
\nc\f{{\mathfrak f}}

\nc{\qo}{{\mathfrak q}}
\nc{\po}{{\mathfrak p}}
\nc{\s}{{\mathfrak s}}
\nc\w{\text{w}}

\renewcommand{\mod}{{\on{-mod}}}

\nc\mathi\iota
\nc\Spec{\on{Spec}}
\nc\Proj{\on{Proj}}
\nc\Mod{\on{Mod}}
\nc{\tw}{\widetilde{\mathfrak t}}
\nc{\pw}{\widetilde{\mathfrak p}}
\nc{\qw}{\widetilde{\mathfrak q}}
\nc{\jw}{\widetilde j}

\nc{\grb}{\overline{\Gr}}
\nc{\I}{\mathcal I}

\nc{\lambdach}{{\check\lambda}}
\nc{\Lambdach}{{\check\Lambda}{}}
\nc{\much}{{\check\mu}}
\nc{\omegach}{{\check\omega}}
\nc{\nuch}{{\check\nu}}
\nc{\etach}{{\check\eta}}
\nc{\alphach}{{\check\alpha}}
\nc{\oblvtach}{{\check\oblvta}}
\nc{\rhoch}{{\check\rho}}
\nc{\ch}{{\check h}}

\nc{\Hb}{\overline{\H}}

\nc{\BA}{{\mathbb{A}}}
\nc{\BC}{{\mathbb{C}}}
\nc{\BG}{{\mathbb{G}}}
\nc{\BM}{{\mathbb{M}}}
\nc{\BO}{{\mathbb{O}}}
\nc{\BD}{{\mathbb{D}}}
\nc{\BN}{{\mathbb{N}}}
\nc{\BP}{{\mathbb{P}}}
\nc{\BQ}{{\mathbb{Q}}}
\nc{\BR}{{\mathbb{R}}}
\nc{\BZ}{{\mathbb{Z}}}
\nc{\BS}{{\mathbb{S}}}
\nc{\Deep}{{\bf{deep}}}
\nc{\deep}{deep}

\nc{\CA}{{\mathcal{A}}}
\nc{\CB}{{\mathcal{B}}}

\nc{\CE}{{\mathcal{E}}}
\nc{\CF}{{\mathcal{F}}}
\nc{\CH}{{\mathcal{H}}}

\nc{\CL}{{\mathcal{L}}}
\nc{\CC}{{\mathcal{C}}}
\nc{\CG}{{\mathcal{G}}}
\nc{\CalD}{{\mathcal{D}}}
\nc{\CM}{{\mathcal{M}}}
\nc{\CN}{{\mathcal{N}}}
\nc{\CK}{{\mathcal{K}}}
\nc{\CO}{{\mathcal{O}}}
\nc{\CP}{{\mathcal{P}}}
\nc{\CQ}{{\mathcal{Q}}}
\nc{\CR}{{\mathcal{R}}}
\nc{\CS}{{\mathcal{S}}}
\nc{\CT}{{\mathcal{T}}}
\nc{\CU}{{\mathcal{U}}}
\nc{\CV}{{\mathcal{V}}}
\nc{\CW}{{\mathcal{W}}}
\nc{\CX}{{\mathcal{X}}}
\nc{\CY}{{\mathcal{Y}}}
\nc{\CZ}{{\mathcal{Z}}}
\nc{\CI}{{\mathcal{I}}}

\nc{\csM}{{\check{\mathcal A}}{}}
\nc{\oM}{{\overset{\circ}{\mathcal M}}{}}
\nc{\obM}{{\overset{\circ}{\mathbf M}}{}}
\nc{\oCA}{{\overset{\circ}{\mathcal A}}{}}
\nc{\obA}{{\overset{\circ}{\mathbf A}}{}}
\nc{\ooM}{{\overset{\circ}{M}}{}}
\nc{\osM}{{\overset{\circ}{\mathsf M}}{}}
\nc{\vM}{{\overset{\bullet}{\mathcal M}}{}}
\nc{\nM}{{\underset{\bullet}{\mathcal M}}{}}
\nc{\oD}{{\overset{\circ}{\mathcal D}}{}}
\nc{\obD}{{\overset{\circ}{\mathbf D}}{}}
\nc{\oA}{{\overset{\circ}{\mathbb A}}{}}
\nc{\op}{{\overset{\bullet}{\mathbf p}}{}}
\nc{\cp}{{\overset{\circ}{\mathbf p}}{}}
\nc{\oU}{{\overset{\bullet}{\mathcal U}}{}}
\nc{\oZ}{{\overset{\circ}{\mathcal Z}}{}}
\nc{\ofZ}{{\overset{\circ}{\mathfrak Z}}{}}
\nc{\oF}{{\overset{\circ}{\fF}}}

\nc{\fa}{{\mathfrak{a}}}
\nc{\fb}{{\mathfrak{b}}}
\nc{\fd}{{\mathfrak{d}}}
\nc{\ff}{{\mathfrak{f}}}
\nc{\fg}{{\mathfrak{g}}}
\nc{\fgl}{{\mathfrak{gl}}}
\nc{\fh}{{\mathfrak{h}}}
\nc{\fj}{{\mathfrak{j}}}
\nc{\fk}{{\mathfrak{k}}}
\nc{\fl}{{\mathfrak{l}}}
\nc{\fm}{{\mathfrak{m}}}
\nc{\fn}{{\mathfrak{n}}}
\nc{\fu}{{\mathfrak{u}}}
\nc{\fp}{{\mathfrak{p}}}
\nc{\fr}{{\mathfrak{r}}}
\nc{\fs}{{\mathfrak{s}}}
\nc{\ft}{{\mathfrak{t}}}
\nc{\fv}{{\mathfrak{v}}}
\nc{\fz}{{\mathfrak{z}}}
\nc{\fsl}{{\mathfrak{sl}}}
\nc{\hsl}{{\widehat{\mathfrak{sl}}}}
\nc{\hgl}{{\widehat{\mathfrak{gl}}}}
\nc{\hg}{{\widehat{\mathfrak{g}}}}
\nc{\chg}{{\widehat{\mathfrak{g}}}{}^\vee}
\nc{\hn}{{\widehat{\mathfrak{n}}}}
\nc{\chn}{{\widehat{\mathfrak{n}}}{}^\vee}

\nc{\fA}{{\mathfrak{A}}}
\nc{\fB}{{\mathfrak{B}}}
\nc{\fD}{{\mathfrak{D}}}
\nc{\fE}{{\mathfrak{E}}}
\nc{\fF}{{\mathfrak{F}}}
\nc{\fG}{{\mathfrak{G}}}
\nc{\fK}{{\mathfrak{K}}}
\nc{\fL}{{\mathfrak{L}}}
\nc{\fM}{{\mathfrak{M}}}
\nc{\fN}{{\mathfrak{N}}}
\nc{\fP}{{\mathfrak{P}}}
\nc{\fU}{{\mathfrak{U}}}
\nc{\fV}{{\mathfrak{V}}}
\nc{\fZ}{{\mathfrak{Z}}}

\nc{\bb}{{\mathbf{b}}}
\nc{\bc}{{\mathbf{c}}}
\nc{\bd}{{\mathbf{d}}}
\nc{\bbf}{{\mathbf{f}}}
\nc{\be}{{\mathbf{e}}}
\nc{\bi}{{\mathbf{i}}}
\nc{\bj}{{\mathbf{j}}}
\nc{\bn}{{\mathbf{n}}}
\nc{\bo}{{\mathbf{o}}}
\nc{\bp}{{\mathbf{p}}}
\nc{\bq}{{\mathbf{q}}}
\nc{\bu}{{\mathbf{u}}}
\nc{\bv}{{\mathbf{v}}}
\nc{\bx}{{\mathbf{x}}}
\nc{\bs}{{\mathbf{s}}}
\nc{\by}{{\mathbf{y}}}
\nc{\bw}{{\mathbf{w}}}
\nc{\bA}{{\mathbf{A}}}
\nc{\bK}{{\mathbf{K}}}
\nc{\bB}{{\mathbf{B}}}
\nc{\bC}{{\mathbf{C}}}
\nc{\bG}{{\mathbf{G}}}
\nc{\bD}{{\mathbf{D}}}
\nc{\bH}{{\mathbf{H}}}
\nc{\bM}{{\mathbf{M}}}
\nc{\bN}{{\mathbf{N}}}
\nc{\bO}{{\mathbf{O}}}
\nc{\bV}{{\mathbf{V}}}
\nc{\bW}{{\mathbf{W}}}
\nc{\bX}{{\mathbf{X}}}
\nc{\bZ}{{\mathbf{Z}}}
\nc{\bS}{{\mathbf{S}}}

\nc{\sA}{{\mathsf{A}}}
\nc{\sB}{{\mathsf{B}}}
\nc{\sC}{{\mathsf{C}}}
\nc{\sD}{{\mathsf{D}}}
\nc{\sF}{{\mathsf{F}}}
\nc{\sG}{{\mathsf{G}}}
\nc{\sH}{{\mathsf{H}}}
\nc{\sK}{{\mathsf{K}}}
\nc{\sk}{{\mathsf{k}}}
\nc{\sM}{{\mathsf{M}}}
\nc{\sO}{{\mathsf{O}}}
\nc{\sW}{{\mathsf{W}}}
\nc{\sQ}{{\mathsf{Q}}}
\nc{\sP}{{\mathsf{P}}}
\nc{\sR}{{\mathsf{R}}}
\nc{\sZ}{{\mathsf{Z}}}
\nc{\sfp}{{\mathsf{p}}}
\nc{\sfq}{{\mathsf{q}}}
\nc{\sr}{{\mathsf{r}}}
\nc{\bk}{{\mathsf{k}}}
\nc{\sg}{{\mathsf{g}}}
\nc{\sff}{{\mathsf{f}}}
\nc{\sfb}{{\mathsf{b}}}
\nc{\sfc}{{\mathsf{c}}}
\nc{\sd}{{\mathsf{d}}}

\nc{\BK}{{\bar{K}}}

\nc{\tA}{{\widetilde{\mathbf{A}}}}
\nc{\tB}{{\widetilde{\mathcal{B}}}}
\nc{\tg}{{\widetilde{\mathfrak{g}}}}
\nc{\tG}{{\widetilde{G}}}
\nc{\TM}{{\widetilde{\mathbb{M}}}{}}
\nc{\tO}{{\widetilde{\mathsf{O}}}{}}
\nc{\tU}{{\widetilde{\mathfrak{U}}}{}}
\nc{\TZ}{{\tilde{Z}}}
\nc{\tx}{{\tilde{x}}}
\nc{\tbv}{{\tilde{\bv}}}
\nc{\tfP}{{\widetilde{\mathfrak{P}}}{}}
\nc{\tz}{{\tilde{\zeta}}}
\nc{\tmu}{{\tilde{\mu}}}

\nc{\urho}{\underline{\rho}}
\nc{\uB}{\underline{B}}
\nc{\uC}{{\underline{\mathbb{C}}}}
\nc{\ui}{\underline{i}}
\nc{\uj}{\underline{j}}
\nc{\ofP}{{\overline{\mathfrak{P}}}}
\nc{\oB}{{\overline{\mathcal{B}}}}
\nc{\og}{{\overline{\mathfrak{g}}}}
\nc{\oI}{{\overline{I}}}

\nc{\eps}{\varepsilon}
\nc{\hrho}{{\hat{\rho}}}

\nc{\one}{{\mathbf{1}}}
\nc{\two}{{\mathbf{t}}}

\nc{\Rep}{{\mathop{\operatorname{\rm Rep}}}}
\nc{\Tot}{{\mathop{\operatorname{\rm Tot}}}}
\nc{\Ker}{{\mathop{\operatorname{\rm Ker}}}}
\nc{\im}{{\mathop{\operatorname{\rm Im}}}}
\nc{\Hilb}{{\mathop{\operatorname{\rm Hilb}}}}
\nc{\End}{{\mathop{\operatorname{\rm End}}}}
\nc{\Ext}{{\mathop{\operatorname{\rm Ext}}}}
\nc{\CHom}{{\mathop{\operatorname{{\mathcal{H}}\it om}}}}
\nc{\GL}{{\mathop{\operatorname{\rm GL}}}}
\nc{\gr}{{\mathop{\operatorname{\rm gr}}}}
\nc{\HN}{{\mathop{\operatorname{\rm HN}}}}
\nc{\Id}{{\mathop{\operatorname{\rm Id}}}}
\nc{\de}{{\mathop{\operatorname{\rm def}}}}
\nc{\length}{{\mathop{\operatorname{\rm length}}}}
\nc{\supp}{{\mathop{\operatorname{\rm supp}}}}

\nc{\Cliff}{{\mathsf{Cliff}}}
\nc{\Fl}{\on{Fl}}
\nc{\Fib}{{\mathsf{Fib}}}
\nc{\Coh}{{\on{Coh}}}
\nc{\QCoh}{{\on{QCoh}}}
\nc{\IndCoh}{{\on{IndCoh}}}
\nc{\FCoh}{{\mathsf{FCoh}}}

\nc{\reg}{{\text{\rm reg}}}

\nc{\cplus}{{\mathbf{C}_+}}
\nc{\cminus}{{\mathbf{C}_-}}
\nc{\cthree}{{\mathbf{C}_\bullet}}
\nc{\Qbar}{{\bar{Q}}}
\nc\Eis{\on{Eis}}
\nc\Eisb{\ol\Eis{}}
\nc\Eisr{\on{Eis}^{rat}{}}
\nc\wh{\widehat}
\nc{\Def}{\on{Def_{\check{\fb}}(E)}}
\nc{\barZ}{\overline{Z}{}}
\nc{\barbarZ}{\overline{\barZ}{}}
\nc{\barpi}{\overline\pi}
\nc{\barbarpi}{\overline\barpi}
\nc{\barpip}{\overline\pi{}^+}
\nc{\barpim}{\overline\pi{}^-}

\nc{\fq}{\mathfrak q}

\nc{\fqb}{\ol{\sfq}{}}
\nc{\fpb}{\ol{\sfp}{}}
\nc{\fpr}{{\sfp^{rat}}{}}
\nc{\fqr}{{\sfq^{rat}}{}}

\nc{\hattimes}{\wh\otimes}

\nc{\bh}{{\bar{h}}}
\nc{\bOmega}{{\overline{\Omega(\check \fn)}}}

\nc{\seq}[1]{\stackrel{#1}{\sim}}

\nc{\cT}{{\check{T}}}
\nc{\cG}{{\check{G}}}
\nc{\cM}{{\check{M}}}
\nc{\cB}{{\check{B}}}

\nc{\ct}{{\check{\mathfrak t}}}
\nc{\cg}{{\check{\fg}}}
\nc{\cb}{{\check{\fb}}}
\nc{\cn}{{\check{\fn}}}

\nc{\cLambda}{{\check\Lambda}}

\nc{\cla}{{\check\lambda}}
\nc{\cmu}{{\check\mu}}
\nc{\cnu}{{\check\nu}}
\nc{\ceta}{{\check\eta}}

\nc{\DefbE}{{\on{Def}_{\cB}(E_\cT)}}

\nc{\imathb}{{\ol{\imath}}}
\nc{\rlr}{\overset{\longrightarrow}{\underset{\longrightarrow}\longleftarrow}}

\nc{\oBun}{\overset{\circ}\Bun}
\nc{\LocSys}{\on{LocSys}}
\nc{\BunBbb}{\ol{\ol{Bun}}_B}
\nc{\BunBr}{\Bun_B^{rat}}
\nc{\BunBrsg}{\Bun_B^{rat,\on{s.g.}}}
\nc{\BunBrp}{\Bun_B^{rat,polar}}
\nc{\BunBrpbg}{\Bun_B^{rat,polar,\on{b.g.}}}
\nc{\BunBrpsg}{\Bun_B^{rat,polar,\on{s.g.}}}
\nc{\BunTrp}{\Bun_T^{rat,polar}}
\nc{\BunTrpbg}{\Bun_T^{rat,polar,\on{b.g.}}}
\nc{\BunTrpsg}{\Bun_T^{rat,polar,\on{s.g.}}}
\nc{\BunNr}{\Bun_N^{rat}}
\nc{\BunNre}{\Bun_N^{enh,rat}}
\nc{\BunTr}{\Bun_T^{rat}}
\nc{\Vect}{\on{Vect}}
\nc{\Whit}{\on{Whit}}
\nc{\CTb}{\ol{\on{CT}}}
\nc{\Ran}{\on{Ran}}
\nc{\CTr}{\on{CT}^{rat}{}}
\nc\jmathr{\jmath^{rat}{}}
\nc{\ux}{\underline{x}}
\nc{\clambda}{{\check\lambda}}
\nc{\calpha}{{\check\alpha}}
\nc{\ind}{{\mathbf{ind}}}
\nc{\oblv}{{\mathbf{oblv}}}
\nc{\ox}{{\overline{x}}}
\nc{\cLa}{\check{\Lambda}}
\nc{\StinftyCat}{\on{DGCat}}
\nc{\inftyCat}{\infty\on{-Cat}}
\nc{\inftygroup}{\infty\on{-Grpd}}
\nc{\Dmod}{\on{D-mod}}
\nc{\CMaps}{{\mathcal Maps}}
\nc{\Maps}{\on{Maps}}
\nc{\affSch}{\on{Sch}^{\on{aff}}}
\nc{\dr}{{\on{dR}}}
\nc{\oCY}{\overset{\circ}\CY}
\nc{\leqG}{\underset{G}\leq}
\nc{\leqM}{\underset{M}\leq}
\nc{\leqGad}{\underset{G_{ad}}\leq}
\nc{\leqMad}{\underset{M_{ad}}\leq}
\nc{\Tr}{\on{Tr}}
\nc{\Frob}{\on{Frob}}
\nc{\DGCat}{\on{DGCat}}
\nc{\tDGCat}{2\on{-DGCat}}
\nc{\ev}{\on{ev}}
\nc{\mmod}{\on{-}\mathbf{mod}}
\nc{\sotimes}{\overset{!}\otimes}
\nc{\Sht}{{\on{Sht}}}
\nc{\Res}{{\on{Res}}}
\nc{\Av}{{\on{Av}}}
\nc{\Ind}{{\on{Ind}}}
\nc{\coInd}{{\on{coInd}}}
\nc{\ul}{\underline}
\nc{\triv}{\mathbf{triv}}

\title[On the Casselman-Jacquet functor]{On the Casselman-Jacquet functor}

\author{T.-H.~Chen, D.~Gaitsgory and A.~Yom Din}

\begin{document}

\date{\today}

\dedicatory{Dedicated to J.~Bernstein} 

\begin{abstract}
We study the Casselman-Jacquet functor $J$, viewed as a functor from the (derived) category of $(\fg,K)$-modules
to the (derived) category of $(\fg,N^-)$-modules, $N^-$ is the negative maximal unipotent. We give
a functorial definition of $J$ as a certain right adjoint functor, and identify it as a composition of two averaging functors
$\Av^{N^-}_!\circ \Av^N_*$. We show that it is also isomorphic to the composition $\Av^{N^-}_*\circ \Av^N_!$. Our key tool
is the \emph{pseudo-identity} functor that acts on the (derived) category of (twisted) $D$-modules on an algebraic stack. 
\end{abstract}  

\maketitle

\tableofcontents

\section*{Introduction}

\ssec{The Casselman-Jacquet functor}

\sssec{}

Let $G$ be a real reductive algebraic group, $\fg$ the complexification of $\text{Lie}(G)$, $K$ the complexification of a maximal compact subgroup in $G(\mathbb{R})$, and $\fn , \fn^- \subset \fg$ the complexifications of the Lie algebras of the unipotent radicals of opposite minimal parabolics in $G$. 
Let $(\fg,K)\mod$ denote the corresponding category of $(\fg,K)$-modules. Recall that a Harish-Chandra module is a $(\fg,K)$-module that
is of finite length (equivalently, finitely generated and acted on locally finitely by the center of $U(\fg)$). 

\medskip

In his work on representations of real reductive groups, W.A.~Casselman introduced a remarkable functor on the category of
Harish-Chandra modules: it is defined by the formula
\begin{equation} \label{e:CJ}
\CM\mapsto \wh{J}(\CM):=\underset{k}{\underset{\longleftarrow}{\on{lim}}}\, \CM/\fn^k\cdot M,
\end{equation} 
where $\fn$ is the unipotent radical of a minimal parabolic.

\medskip

A key property of the functor $\wh{J}$ is that it is \emph{exact} and \emph{conservative}; this provided a new tool for the study of 
the category of Harish-Chandra modules, leading to an array of powerful results. 

\sssec{}

The functor \eqref{e:CJ} has an algebraic cousin, denoted $J$, and defined as follows. 

\medskip

Pick a cocharacter $\BG_m \to G$ that is dominant and regular in the split Cartan,  
and let $J(\CM)$ be the subset of $\wh{J}(\CM)$
on which $\BA^1=\on{Lie}(\BG_m)$ acts locally finitely, i.e., the direct sum of generalized eigenspaces with respect to the generator
$t\in \BA^1$:
$$J(\CM)\simeq \underset{\lambda}\oplus\, J(\CM)_\lambda:=\underset{\lambda}\oplus\, \wh{J}(\CM)_\lambda.$$

One shows that the entire $\wh{J}(\CM)$ can be recovered as 
$$\underset{\lambda}\Pi\, J(\CM)_\lambda,$$
so the information contained in $J$ is more or less equivalent to that possessed by $\wh{J}$. In particular,
the functor $J$ is also exact. 

\medskip

We will refer to $J$ as the \emph{Casselman-Jacquet} functor. 

\sssec{}  \label{sss:n- finite}

An important feature of the functor $J$, and one relevant to this paper, is that it can be extracted from $\wh{J}$
using the Lie algebra $\fn^-$ (the unipotent radical of the opposite parabolic) rather than the split Cartan.

\medskip

Namely, an elementary argument shows that $J(\CM)$ can be identified with the
subset of vectors in $\wh{J}(\CM)$ on which $\fn^-$ acts locally nilpotently.

\medskip

Thus, we can think of $J$ as a functor
$$(\fg,K)\mod_\chi \to (\fg,M_K\cdot N^-)\mod_\chi,$$
where our notations are as follows:

\begin{itemize}

\item $(\fg,K)\mod$ denotes the abelian category of $(\fg,K)$-modules ($K$ is the algebraic group corresponding to the maximal compact);

\item $(\fg,M_K\cdot N^-)\mod$ denotes the abelian category of $(\fg,M_K\cdot N^-)$-modules 
($N^-$ and $M_K$ are the algebraic groups corresponding to the
(opposite) maximal unipotent and compact part of the Levi, respectively);

\item The subscript $\chi$ indicates that we are considering categories of modules with a fixed central character $\chi$.

\end{itemize} 

\sssec{Our goals}

The primary goals of the present paper are as follows:

\medskip

\noindent--Extend the definitions of the functors $\wh{J}$ and $J$ from the abelian categories to the corresponding derived categories
$\fg\mod^K_\chi$, $\fg\mod^{M_K\cdot N^-}_\chi$, etc., (and in particular, explain their functorial meaning);

\medskip

\noindent--Express the functor $J$ as a double-averaging functor, and thus reprove the corresponding result from the paper \cite{CY}, where it
was obtained by interpreting $J$ via nearby cycles using \cite{ENV};

\medskip

\noindent--Record a conjecture that states that the functor $J$ is (up to some twist) the \emph{right} adjoint of the functor of \emph{averaging}
with respect to $K$:
$$\on{Av}^{K/M_K}_*:\fg\mod^{M_K\cdot N^-}_\chi\to \fg\mod^K_\chi,$$
and explain that this is analogous to Bernstein's ``2nd adjointness theorem" for $\fp$-adic groups. 

\sssec{}

In the course of realizing these goals we will encounter another operation of interest: Drinfeld's pseudo-identity functor 
on the category of $M_K\cdot N$-equivariant (twisted) D-modules on the flag variety $X$. 

\medskip

This functor will be used in the
proofs of the main results, and as such, may seem to be not more than a trick. However, in the sequel to this paper 
it will be explained that this functor plays a conceptual role at the categorical level. 

\ssec{Functorial interpretation of the Casselman-Jacquet functor}

\sssec{}

We first give a functorial interpretation of the (derived version of the) functor $\wh{J}$. 

\medskip

Namely, in Sects. \ref{ss:completion} and \ref{ss:ident}, we show that (the derived) 
version of this functor identifies with the composition
$$\fg\mod_\chi \overset{\Av^N_*}\longrightarrow \fg\mod_\chi^N \overset{(\Av^N_*)^R}\longrightarrow \fg\mod_\chi.$$

Here $\Av^N_*$ is the functor of *-averaging with respect to $N$, i.e., the right adjoint to the forgetful functor
\begin{equation} \label{e:intro oblv}
\oblv_N:\fg\mod_\chi^N\to \fg\mod_\chi,
\end{equation} 
and $(\Av^N_*)^R$ is the (a priori, discontinuous) \emph{right} adjoint of $\Av^N_*$. 

\medskip

In fact, we show this in a rather general situation when instead of $\fg\mod_\chi$ we consider the category $A\mod$,
where $A$ is an associative algebra, equipped with a \emph{Harish-Chandra structure} with respect to $N$ (see
\secref{sss:H-Ch algebra} for what this means). 

\sssec{}

Next, we consider the functor $J$, in the general setting of a category $\CC$ equipped with an action of $G$
(for example, for $\CC=A\mod$ for an associative algebra $A$, equipped with a Harish-Chandra structure
with respect to all of $G$), see \secref{sss:action} for what this means. 

\medskip

We define the (derived version of the) functor $J$ as the composite
$$J:=\Av^{N^-}_*\circ (\Av^N_*)^R\circ \Av^N_*: \CC^{M_K}\to \CC^{M_K\cdot N^-}.$$

Assume that the following property is satisfied (which is the case for $\CC=\fg\mod_\chi$ or $\CC$ being the category $\Dmod_\lambda(X)$
of twisted D-modules on the flag variety):

\medskip

\noindent(*) {\it The ``long intertwining functor" 
\begin{equation} \label{e:intro Ups}
\Upsilon:=\Av^N_*\circ \oblv_{N^-}:\CC^{M_K\cdot N^-}\to \CC^{M_K\cdot N},
\end{equation} 
given by forgetting $N^-$-equivariance and then averaging with respect to $N$, is an equivalence.}

\medskip

In this case we show (see \propref{p:main} and its variant in \secref{sss:modify}) that we have a canonical isomorphism of functors
\begin{equation} \label{e:J !*}
J\simeq \Av^{N^-}_!\circ \Av^N_*.
\end{equation} 

In the above formula, $\Av^{N^-}_!$ is the !-averaging functor with respect to $N^-$, i.e., the \emph{left} adjoint to
\eqref{e:intro oblv} (with $N$ replaced by $N^-$). 

\medskip

The isomorphism \eqref{e:J !*} had been initially obtained in \cite{CY}; we will comment on that in \secref{sss:CY} below. 

\sssec{}

Finally, in the particular case of $\CC=\fg\mod_\chi$ we will show (see \thmref{t:J self-dual on g}
and its variant in \secref{sss:modify}) that $J$ is canonically isomorphic to
its Verdier dual functor
\begin{equation} \label{e:J *!}
J\simeq \Av^{N^-}_*\circ \Av^N_!,
\end{equation} 
when applied to objects in $\fg\mod^{M_K}_\chi$ whose cohomologies are finitely generated over $\fn$ (or are direct limits of such).

\medskip

We will deduce \eqref{e:J *!} from a similar statement for $\CC=\Dmod_\lambda(X)$  (see \thmref{t:J self-dual on X} and its variant in \secref{sss:modify}),
where the isomorphism in question holds for objects on $\Dmod_\lambda(X)^{M_K}$ that are ULA with respect to the projection $X\to N\backslash X$ 
(see \secref{sss:ULA} for what this means). 

\medskip

The isomorphism \eqref{e:J *!} implies that the functor $J$ is t-exact when applied to $\fg\mod_\chi^K$. In particular, it shows that the procedure 
in \secref{sss:n- finite} \emph{does not omit higher cohomologies} (in principle, the functor of taking $\fn^-$-locally nilpotent vectors 
should be derived). 

\sssec{}  \label{sss:CY}

The isomorphism \eqref{e:J *!}, applied to objects from $\fg\mod_\chi^K$, 
had been obtained in \cite{CY} as a combination of the following two results:

\medskip

\noindent--One is the main theorem of \cite{ENV} that shows that under the localization equivalence (for this one assumes that $\chi$ is regular),
the functor $J$ (defined as in \secref{sss:n- finite}) corresponds to a certain nearby cycles functor 
$$\Psi:\Dmod_\lambda(X)^K\to \Dmod_\lambda(X)^{M_K\cdot N^-}.$$
This was done by explicitly analyzing the V-filtration on the corresponding D-module. 

\medskip

\noindent--The other is the key result from the paper \cite{CY} itself, which establishes an isomorphism
$$\Psi\simeq \Av^{N^-}_*\circ \Av^N_!,$$
again on $\Dmod_\lambda(X)^K$;

\medskip

The isomorphism \eqref{e:J !*} was then deduced from \eqref{e:J *!} using the Verdier self-duality property of the nearby cycles functor 
$\Psi$, which implies that
$$\Av^{N^-}_*\circ \Av^N_!\simeq \Psi\simeq \Av^{N^-}_!\circ \Av^N_*.$$

\medskip

So, the present paper gives another, in a sense more direct proof of \eqref{e:J !*} and \eqref{e:J *!}, which does not appeal to the
nearby cycles functor (however, we do \emph{not} to imply that the latter is irrelevant: see \secref{sss:rel to BK}). 

\ssec{The pseudo-identity functor and the ULA condition}

\sssec{}

The pseudo-identity functor is a certain canonical endofunctor of the category of (twisted) D-modules on any algebraic stack, denoted
$$\on{Ps-Id}_\CY:\Dmod_\lambda(\CY)\to \Dmod_\lambda(\CY),$$
see \secref{ss:Psi}. Its definition was suggested by V.~Drinfeld and was recorded in \cite{Ga1}. 

\medskip

This functor is \emph{uninteresting} (equals to the identity functor up to a shift) when $\CY$ is a smooth separated scheme,
but has some very interesting properties on certain algebraic stacks that appear in geometric representation theory, see, e.g., 
\cite{Ga3}. 

\sssec{}

In this paper we apply this functor to the stack $\CY$ equal to $H\backslash X$, where $X$ a proper scheme acted on
by an algebraic group $H$ (in our applications, $X$ will be the flag variety of $G$ and $H=M_K\cdot N$). 

\medskip

We prove the following result (see \thmref{t:pseudo-id and av}):

\begin{thm}  \label{t:intertw}
Let $f$ denote the projection $X\to H\backslash X$. Then the functors
$$f_![2\dim(H)] \text{ and } \on{Ps-Id}_{H\backslash X}\circ f_*[2\dim(X)]$$
are canonically isomorphic when evaluated on objects of $\Dmod_\lambda(X)$ that are ULA with respect to $f$.
\end{thm}

In other words, this theorem says that the functor $\on{Ps-Id}_{H\backslash X}$ intertwines the !- and *- direct images along $f$.

\sssec{}

From \thmref{t:intertw} we deduce:

\begin{cor} \label{c:intertw} For $X$ being the flag variety of $G$,  
the functor $\on{Ps-Id}_{M_K\cdot N\backslash X}$ induces a self-equivalence of $\Dmod_\lambda(M_K\cdot N\backslash X)$; moreover, this
self-equivalence is canonically isomorphic (up to a shift) to the composition
$$\Dmod_\lambda(M_K\cdot N\backslash X)\overset{\Upsilon^{-1}}\longrightarrow 
\Dmod_\lambda(M_K\cdot N^-\backslash X) \overset{(\Upsilon^{-})^{-1}}\longrightarrow \Dmod_\lambda(M_K\cdot N\backslash X).$$
where $\Upsilon$ is the long intertwining functor of \eqref{e:intro Ups}, and $\Upsilon^-$ is the analogous functor where the roles of $N$ and $N^-$ are swapped.
\end{cor} 

\sssec{}

The application of the functor $\on{Ps-Id}_{M_K\cdot N\backslash X}$ in this paper is the following one: 

\medskip

Combining \corref{c:intertw} and \thmref{t:intertw} we obtain an isomorphism of functors
\begin{equation} \label{e:self-dual intro X}
\Av^{N^-}_*\circ \Av^N_! \text{ and } \Av^{N^-}_!\circ \Av^N_*:
\Dmod_\lambda(M_K\backslash X)\rightrightarrows \Dmod_\lambda(M_K\cdot N^-\backslash X),
\end{equation} 
on the subcategory objects of $\Dmod_\lambda(M_K\backslash X)$ that are ULA with respect to the
projection
$M_K\backslash X\to M_K\cdot N\backslash X$ (or are direct limits of such).

\sssec{}

Finally, let us comment on the relationship between the isomorphism of functors \eqref{e:self-dual intro X} and the isomorphism
\begin{equation} \label{e:self-dual intro g}
\Av^{N^-}_*\circ \Av^N_! \text{ and } \Av^{N^-}_!\circ \Av^N_*:\fg\mod_\chi^{M_K}\rightrightarrows \fg\mod_\chi^{M_K\cdot N^-},
\end{equation} 
on the subcategory consisting of modules which are finitely generated over $\fn$ (or are direct limits of such).

\medskip

The point is that the isomorphisms \eqref{e:self-dual intro X} and \eqref{e:self-dual intro g} are logically equivalent, 
using the following observation (\propref{p:Loc and ULA}): 

\medskip

\noindent The functors
$$\on{Loc}: \fg\mod_\chi\rightleftarrows \Dmod_\lambda(X):\Gamma(X,-)$$
map the corresponding subcategories to one-another. 

\ssec{The ``2nd adjointness" conjecture}

\sssec{}

Let us consider the ``principal series" functor in the context of $(\fg,K)$-modules. We stipulate this to be the functor
$$\fg\mod_\chi^{M_K\cdot N^-}\overset{\oblv_{N^-}}\longrightarrow \fg\mod_\chi^{M_K} 
\overset{\Av^{K/M_K}_*}\longrightarrow \fg\mod_\chi^K.$$

\medskip

Tautologically, this functor is the \emph{right} adjoint of the functor
$$\fg\mod_\chi^K \overset{\oblv^{K/M_K}_*}\longrightarrow \fg\mod_\chi^{M_K}  \overset{\Av^{N^-}_!}\longrightarrow \fg\mod_\chi^{M_K\cdot N^-}.$$

A priori, it is not clear clear that $\Av^{K/M_K}_*\circ \oblv_{N^-}$ itself should admit a \emph{right} adjoint given by a nice formula.
However, based on the analogy with Bernstein's 2nd adjointness theorem (see \secref{ss:2nd adj}) we propose the following
conjecture:

\begin{conj} \label{c:intro 2nd adj}
The functor $J\circ \oblv_{K/M_K}$ (up to a shift) provides a right adjoint to the principal series functor
$\Av^{K/M_K}_*\circ \oblv_{N^-}$. 
\end{conj}

In the sequel to this paper further evidence towards the validity of 
\conjref{c:intro 2nd adj} will be provided, and the logical equivalence 
between \conjref{c:intro 2nd adj} and \cite[Conjectures 9.1.4, 9.1.6]{Yo} 
will be explained. 

\medskip

In addition to \conjref{c:intro 2nd adj}, we make a similar conjecture when the category $\fg\mod_\chi$ is replaced by $\Dmod_\lambda(X)$.

\sssec{}  \label{sss:rel to BK} 

At the moment, it is not clear to the authors how to write down either the unit or the counit for the conjectural adjunction between 
$J\circ \oblv_{K/M_K}$ or $\Av^{K/M_K}_*\circ \oblv_{N^-}$, either in the context of $\fg\mod_\chi$ or in that of $\Dmod_\lambda(X)$.

\medskip 

The following, however, seems very tempting: 

\medskip

In the paper \cite{BK} it is explained that the in the context of $\fp$-adic groups, Bernstein's 2nd adjointness can be obtained by
analyzing the \emph{wonderful degeneration} of $G$, i.e., the geometry of the \emph{wonderful compactification} near 
the stratum of the boundary corresponding to the given parabolic. 

\medskip

Now, as was mentioned above, one of the main results of \cite{CY} says that the functor $J$ for $\Dmod_\lambda(X)$
is isomorphic to the nearby cycles functor along the same wonderful degeneration. 

\medskip

So it would be very nice to adapt the ideas 
of \cite{BK} to prove \conjref{c:intro 2nd adj}. However, so far, we do not know how to carry this out. 

\ssec{Organization of the paper}

\sssec{}

The main body of the paper starts with \secref{s:recall} where we recall (but also reprove and supply proofs that we could not find
in the literature) the following topics:

\medskip

\noindent--The notion of action of an algebraic group on a (DG) category; the associated notions of equivariance and the *- and !-averaging 
functors;

\medskip

\noindent--The Beilinson-Bernstein localization theory;

\medskip

\noindent--Translation functors;

\medskip

\noindent--The long intertwining functor between $N$- and $N^-$-equivariant categories (either for $\fg$-modules, or D-modules on the
flag variety). 

\medskip

The reader may consider skipping this section on the first pass, and return to it when necessary.

\sssec{}

In \secref{s:J} we initiate the study of the Casselman-Jacquet functor. However, in order to simplify the exposition, in this section
instead of working with a real reductive group (or the corresponding symmetric pair), we work in a completely algebraic situation.

\medskip

I.e., in this section we take $N$ to be a maximal unipotent subgroup in a reductive group $G$, and consider the Casselman-Jacquet
functor $J$ as a functor
$$\fg\mod_\chi\to \fg\mod_\chi^N.$$ 

\noindent (Analogous results in the case of symmetric pairs require very minor modifications, which will be explained in 
\secref{sss:modify}). 

\medskip

\noindent--We define the functor $J$ (in the context of a category $\CC$ acted on by $G$) as the composition
$$\Av^{N^-}_*\circ (\Av^N_*)^R\circ \Av^N_*: \CC \to \CC^{N^-}.$$

\medskip

\noindent--We show that for $\CC=A\mod$ (for an associative algebra $A$ equipped with a Harish--Chandra structure with respect to $G$),
the functor 
$$(\Av^N_*)^R\circ \Av^N_*:A\mod\to A\mod$$
is given by $\fn$-adic completion.

\medskip

\noindent--We show that if the functor 
$$\Upsilon:=\Av^N_*\circ \oblv_{N^-}:\CC^{N^-}\to \CC^N$$
is an equivalence, then $J$ identifies canonically with
$$\Av^{N^-}_!\circ \Av^N_*.$$

\medskip

\noindent--We state that for $\CC=\Dmod_\lambda(X)$, the functor $J$ is canonically isomorphic to $\Av^{N^-}_*\circ \Av^N_!$, when evaluated on
objects that are ULA with respect to $X\to N\backslash X$. 

\medskip

\noindent--From here we deduce the corresponding isomorphism for $\fg\mod_\chi$
(on objects that are finitely generated with respect to $\fn$).

\medskip

\noindent--Finally, we show the equivalence between the ULA and $\fn$-f.g. conditions under the localization functor 
$$\on{Loc}:\fg\mod_\chi\to \Dmod_\lambda(X).$$

\sssec{}

In \secref{s:Psi}, our ostensible goal is to prove the isomorphism
\begin{equation} \label{e:self-duality again}
\Av^{N^-}_!\circ \Av^N_*\simeq \Av^{N^-}_*\circ \Av^N_!
\end{equation}
on objects of $\Dmod_\lambda(X)$ that are ULA with respect to $X\to N\backslash X$. 

\medskip

In order to do this we introduce the pseudo-identity functor $\on{Ps-Id}_\CY$, which is an endofunctor on the category
of twisted D-modules on an algebraic stack $\CY$.

\medskip

We deduce \eqref{e:self-duality again} from a key geometric result, \thmref{t:pseudo-id and av}. 

\sssec{}

In \secref{s:gK} we adapt the results of the preceding sections to the context of a symmetric pair, and thereby
deduce the results announced earlier in the introduction. 

\medskip

Finally, we state our version of the 2nd adjointness conjecture and explain the analogy with the corresponding
assertion (which is a theorem of J.~Bernstein) in the case of $\fp$-adic groups. 

\ssec{Conventions and notation}

\sssec{}

Throughout the paper we will be working over an algebraically closed field $k$ of characteristic $0$, and we let $G$
be a connected reductive group over $k$. Throughout the paper, $X$ will denote the flag variety of $G$. 

\medskip

In Sects. \ref{s:recall}-\ref{s:Psi} we let $N$ be the unipotent radical of a Borel subgroup of $G$, and by $N^-$ 
the unipotent radical of an opposite Borel.

\medskip

In \secref{s:gK} we will change the context, and assume that $G$ is endowed with involution $\theta$;
we let $K:=G^\theta$. Let $P$ be a minimal parabolic compatible with $\theta$; in particular $P^-:=\theta(P)$
is an opposite parabolic. For the duration of \secref{s:gK}, we let $N$ be the unipotent radical
of $P$ and $N^-$ the unipotent radical of $P^-$. 

\sssec{}

This paper will make a (mild) use of higher algebra, in that we will be working with DG categories rather than with
triangulated categories (the reluctant reader can avoid this, see \secref{ss:avoid}). See \cite[Sect. 0.6]{DrGa1}
for a concise summary of the theory of DG categories. 

\medskip

Unless specified otherwise, our DG categories will be assumed \emph{cocomplete}, i.e., contain infinite
direct sums. Similarly, unless specified otherwise, functors between DG categories will be assumed
\emph{continuous}, i.e., preserving infinite direct sums.

\medskip

We denote by $\Vect$ the DG category of chain complexes of vector spaces.

\medskip

For a DG category $\CC$ and $\bc_0,\bc_1\in \CC$ we will denote by $\CHom_\CC(\bc_0,\bc_1)\in \Vect$
the Hom complex between them. 

\medskip

For a DG category $\CC$ we will denote by $\CC^c$ the full (but not cocomplete) subcategory consisting of compact objects. 

\sssec{}

If $\CC$ is endowed with a t-structure, we will denote by $\CC^{\leq 0}$ (resp., $\CC^{\geq 0}$) the subcategory
of connective (resp., coconnective) objects, and by $\CC^\heartsuit=\CC^{\leq 0}\cap \CC^{\geq 0}$ its heart.

\medskip

We will say that a functor between DG categories $\CC_1$ and $\CC_2$, each endowed with a t-structure, is
left t-exact (resp., right t-exact, t-exact) if it sends $\CC_1^{\geq 0}$ to $\CC_2^{\geq 0}$ (resp., $\CC_1^{\leq 0}$ to $\CC_2^{\leq 0}$,
both of the above). 

\sssec{}

For an associative algebra $A$ we will denote by $A\mod$ the corresponding DG category of $A$-modules
(and \emph{not} the abelian category). The same applies to $\fg\mod$ for a Lie algebra $\fg$.

\medskip

For a smooth scheme $Y$, equipped with a twisting $\lambda$, we let $\Dmod_\lambda(Y)$ denote
the DG category of twisted D-modules on $Y$. 

\medskip

For an algebraic group $H$, we denote by $\Rep(H)$ the DG category of $H$-representations. 

\ssec{How to get rid of DG categories?}  \label{ss:avoid}

\sssec{}

Unlike its sequel, in this paper we can make do by working with triangulated categories, rather than derived ones. 

\medskip

In general, the necessity to use DG categories arises for two reasons:

\medskip

\noindent--We perform operations on DG categories (e.g., tensor two DG categories over a monoidal DG category acting on them). 

\medskip

\noindent--We take limits/colimits in a given DG category.

\medskip

Both operations are actually present in this paper, but the general procedures can be replaced by \emph{ad hoc} constructions.

\sssec{}

We will be working with the notion of DG category acted on by an algebraic group $H$; if $\CC$ is such a category, we will be considering
the corresponding category $\CC^H$ of $H$-equivariant objects, equipped with the forgetful functor $\oblv_H:\CC^H\to \CC$.
The passage 
$$\CC\rightsquigarrow (\CC^H,\oblv_H)$$
cannot be intrinsically defined within the world of triangulated categories, and that is why we need DG categories. 

\medskip

However, in our particular situation, $H$ will be unipotent, and $\CC^H$ can be defined as the \emph{full subcategory} of
$\CC$, consisting of $H$-invariant objects. This does make sense at the triangulated level, where we regard $\CC$ as 
a triangulated category equipped with the action of the monoidal triangulated category $\Dmod(H)$. 

\sssec{}

In \secref{s:gK} the DG categories $\CC$ that we consider will themselves arise in the form $\CC=\CC_0^H$ for a
\emph{non-unipotent} $H$. However, this will only occur in the following examples:

\medskip

\noindent{(a)} $\CC_0$ is the category of (twisted) D-modules on a scheme $Y$ acted on by $H$;

\medskip

\noindent{(b)} $\CC_0$ is the category $\fg\mod_\chi$, where $\fg$ is a Lie algebra and $\chi$ is its central
character, and $(\fg,H)$ is a Harish-Chandra pair.

\medskip

In both these examples, there are several ways to define the corresponding category $\CC=\CC_0^H$ 
``by hand". 

\medskip

Note, however, that, typically, in neither of these cases will $\CC$ be the derived category of
the heart of its natural t-structure. 

\sssec{}

Finally, the only limits and colimits procedures that we consider will be indexed by filtered sets (in fact, by $\BN$), 
and they will consist of objects inside the heart of a t-structure. So the limit/colimit objects will stay in the heart.  

\ssec{Acknowledgements}

The second and the third authors would like to thank their teacher J.~Bernstein for many illuminating discussions related
to representations of real reductive groups and Harish-Chandra modules; the current paper is essentially an outcome of these
conversations. The third author would like to thank Sam Raskin for very useful conversations on higher categories.
The first author would like to thank the Max Planck Institute for Mathematics for support, hospitality, and a nice research environment.

\medskip

The research of D.G. has been supported by NSF grant DMS-1063470. The research of T.H.C. was partially supported by NSF grant DMS-1702337.

\section{Recollections}  \label{s:recall}

In this section we recall some facts and constructions pertaining to the notion of action of a group on a DG category,
to the Beilinson-Bernstein localization theory, translations functors, and the long intertwining functor. 

\ssec{Groups acting on categories: a reminder}  

\sssec{}  \label{sss:action}

In this paper we will extensively use the notion of (strong) action of an algebraic group $H$ on a DG category $\CC$;
for the definition see 
\cite[Sect. 10.2]{Ga2} (in the terminology of {\it loc.cit.}, these are categories acted on by $H_{\on{dR}}$). 

\medskip

One of the possible definitions is that such a data is equivalent to that of co-action of the co-monoidal DG category $\Dmod(H)$ on $\CC$,
where the co-monoidal structure on $\Dmod(H)$ is given by !-pullback with respect to the product operation on $H$:
\begin{equation} \label{e:coaction}
\CC\overset{\on{co-act}}\longrightarrow \Dmod(H)\otimes \CC.
\end{equation} 

\medskip

We can regard $\Dmod(H)$ also as a \emph{monoidal} category, with respect to the operation of \emph{convolution}, i.e.,
*-pushforward with respect to the product operation on $H$. If $H$ acts on $\CC$, we obtain also a monoidal action of
$\Dmod(H)$ on $\CC$ by the formula
%\begin{multline*} 
$$\Dmod(H)\otimes \CC \overset{\on{Id}\boxtimes \on{co-act}}\longrightarrow 
\Dmod(H)\otimes \Dmod(H)\otimes \CC \overset{\sotimes \boxtimes \on{Id}}\longrightarrow  
\Dmod(H)\otimes \CC \overset{p_*\otimes \on{Id}} \longrightarrow \Vect\otimes \CC\simeq \CC, $$
%\end{multline*} 
where $p_*$ denotes the pushforward functor $\Dmod(H)\to \Dmod(\on{pt})=\Vect$. 

\medskip

We denote the corresponding monoidal operation by 
$$\CF\in \Dmod(H),\,\bc\in \CC\mapsto \CF\star \bc.$$

\sssec{}  \label{sss:H-Ch algebra}

Here are some examples of groups acting on categories that we will use:

\medskip

\noindent{(i)} Let $H$ act on a scheme/algebraic stack $Y$. Then $H$ acts on $\Dmod(Y)$. 

\medskip

\noindent{(i')} Suppose that $Y$ is equipped with a twisting $\lambda$ (see \cite[Sect. 6]{GR} for what this means) that is $H$-equivariant
(the latter means that the twisting descends to one on the quotient stack $H\backslash Y$). 
Then $H$ acts on the category $\Dmod_\lambda(Y)$.

\medskip

\noindent{(ii)} $H$ acts on the category $\fh\mod$ of modules over its own Lie algebra.

\medskip

\noindent{(ii')} Let $\chi$ be the character of the center $Z(\fh)=U(\fh)^{\on{Ad}_H}\subset Z(U(\fh))$. Then $H$ acts on the category $\fh\mod_\chi$,
the latter being the category of $\fh$-modules on which $Z(\fh)$ acts via $\chi$.

\medskip

\noindent{(iii)} Let $A$ be a (classical) associative algebra, equipped with a Harish-Chandra structure with respect to $H$. I.e., we are given 
an action of $H$ on $A$ by automorphisms, and a map of Lie algebras $\phi:\fh\to A$ such that

\begin{itemize}

\item $\phi$ is $H$-equivariant;

\item The adjoint action of $\fh$ on $A$ (coming from $\phi$) equals the derivative of the given $H$-action on $A$.

\end{itemize}

Then $A\mod$ is acted on by $H$. This example contains examples (ii) and (ii') (and also (i) and (i') for $Y$ affine) as particular cases. 

\medskip

An example of this situation is when $A=U(\fg)$, where $(\fg,H)$ is a Harish-Chandra pair, or $A$ is the quotient of $U(\fg)$
by a central character. 

\sssec{} \label{sss:averaging}

If $\CC$ is acted on by $H$, there is a well-defined category $\CC^H$ of $H$-equivariant objects in $\CC$, equipped with a pair
of adjoint functors
$$\oblv_H:\CC^H\rightleftarrows \CC:\Av^H_*.$$

One way to define $\CC^H$ is as the totalization of the co-Bar co-simplicial category
$$\CC\rightrightarrows \Dmod(H)\otimes \CC...$$
associated with the co-action of $\Dmod(H)$ on $\CC$. Under this identification, $\oblv_H$ is given by evaluation on $0$-simplices.

\medskip

Equivalently, we can define $\CC^H$ as the category of co-modules over the co-monad 
$$\oblv_H\circ \Av^H_*:=k_H\,\star -$$
acting on $\CC$, where $k_H\in \Dmod(H)$ is the \emph{constant sheaf} D-module.

\sssec{}

Note that the functor $\oblv_H$ is not necessarily fully faithful. In fact, the composition
$$\Av^H_*\circ \oblv_H:\CC^H\to \CC^H$$
is given by tensor product with $\on{C}^*_{\on{dR}}(H)$ (de Rham cochains on $H$),
where the unit of the adjunction $\on{Id}\to \Av^H_*\circ \oblv_H$
corresponds to the canonical map $k\to \on{C}^*_{\on{dR}}(H)$.

\medskip

The above implies, among the rest, that $\oblv_H$ \emph{is} fully faithful if $H$ is unipotent.

\sssec{}

The functor $\oblv_H$ does not necessarily admit a left adjoint. Its partially defined left adjoint\footnote{For the terminology of partially defined left adjoints etc. see, for example, appendix A of \cite{DrGa2}.} will be denoted by $\Av^H_!$. Concretely, the means that $\Av^H_!$ is defined on the full subcategory
of $\CC$ consisting of objects $c$ for which the functor
$$\CC^H\to \Vect, \quad c'\mapsto \CHom_\CC(c,\oblv_H(c'))$$
is co-representable. 

\sssec{}

Given two subgroups $H_1\subset H_2$ we will denote by
$$\oblv_{H_2/H_1}:\CC^{H_2}\rightleftarrows \CC^{H_1}:\Av_*^{H_2/H_1}$$
the corresponding adjoint pair of functors. 

\medskip

Similarly, will denote by $\Av_!^{H_2/H_1}$ the partially defined left adjoint to $\oblv_{H_2/H_1}$. 

\sssec{}

When $\CC=\Dmod_\lambda(Y)$ we have a canonical identification 
$$\CC^H=\Dmod_\lambda(H\backslash Y),$$
where $\Av^H_*$ is given by the *-pushforward functor 
$$f_*:\Dmod_\lambda(Y)\to \Dmod_\lambda(H\backslash Y),$$
and hence $\oblv_H$ is given by the *-pullback functor $f^*$. Note that the functor $f^*$ 
is well-defined on all D-modules (and not just holonomic ones) because the morphism $f$ is smooth.

\medskip

Let us be in Example (iii) above with $A=U(\fg)$, where $(\fg,H)$ is a Harish-Chandra pair. Then the corresponding category
$\fg\mod^H$ is by definition the (derived) category of $(\fg,H)$-modules. For $\fg=\fh$ we have
$$\fg\mod^H=\Rep(H),$$
the category of $H$-representations. 

\sssec{}  \label{sss:t-structures}

Suppose that $\CC$ is equipped with a t-structure, so that the co-action functor \eqref{e:coaction}
is t-exact, where $\Dmod(H)$ is taken with respect to the \emph{left} D-module t-structure.

\medskip

Then the co-monad 
$$\bc\mapsto k_H\star \bc$$
is left t-exact. This implies that the category $\CC^H$ carries a t-structure, uniquely characterized by the property that
the forgetful functor $\oblv_H$ is t-exact. 

\medskip

In this case, the functor $\Av^H_*$, being the right adjoint of a t-exact functor, is left t-exact. 

\medskip

We will need the following technical assertion:

\begin{lem}  \label{l:Av and t}
Assume that the t-structure on $\CC$ is left-complete,
i.e., for an object $\bc\in \CC$, the map 
$$\bc\to \underset{n}{\on{lim}}\, \tau^{\geq -n}(\bc)$$
is an isomorphism. Then:

\smallskip

\noindent{\em(a)} The t-structure on $\CC^H$ is also left-complete and for $\bc\in \bC$, the natural map
$$\Av^H_*(\bc)\to \underset{n}{\on{lim}}\, \Av^H_*(\tau^{\geq -n}(\bc))$$
is an isomorphism. 

\smallskip

\noindent{\em(b)} If for an object $\bc\in \bC$, the partially defined functor $\Av_!^H$ 
is defined on every $\tau^{\geq -n}(\bc)$, then it is defined on $\bc$ itself, and the natural map
$$\Av^H_!(\bc)\to \underset{n}{\on{lim}}\, \Av^H_!(\tau^{\geq -n}(\bc))$$
is an isomorphism. 

\smallskip

\noindent{\em(b')} If the t-structure on $\CC$ is compatible with filtered colimits, and 
for an object $\bc\in \bC$, the partially defined functor $\Av_!^H$ 
is defined on every $H^n(\bc)$, then it is defined on $\bc$ itself.

\end{lem} 

\ssec{Localization theory}

\sssec{}  \label{sss:localization general}

Let $G$ be an algebraic group acting on a smooth variety $X$, equipped with a $G$-equivariant twisting $\lambda$.
Let $A$ be an associative algebra equipped with a Harish-Chandra structure with respect to $G$, and let us be given
a map 
\begin{equation} \label{e:map to Sect}
A\to \Gamma(X,\on{D}_\lambda),
\end{equation}
as associative algebras equipped with Harish-Chandra structures with respect to $G$. 

\medskip

Then the functor
$$\Gamma(X,-):\Dmod_\lambda(X)\to \Vect$$
naturally factors as 
$$\Dmod_\lambda(X)\to \Gamma(X,\on{D}_\lambda)\mod\to A\mod \to \Vect.$$

By a slight abuse of notation, we will denote the resulting functor 
$$\Dmod_\lambda(X)\to A\mod$$
by the same symbol $\Gamma$. It admits a left adjoint, denoted $\on{Loc}$. Both these functors 
are compatible with the $G$-actions.  

\medskip

Note that the functor $\on{Loc}$ is fully faithful if and only if the map \eqref{e:map to Sect}
is an isomorphism. 

\sssec{}  \label{sss:BB}

The example of this situation of interest for us is, of course, when $G$ is reductive and $X$ is the flag variety of $G$.
In this case $G$-equivariant twistings on $X$ are in bijection with elements of $\ft^*$ (the dual vector space of the
abstract Cartan $\ft$), where we take the $\rho$-shift into account. 

\medskip

Given $\lambda\in \ft^*$, we take
$$A=U(\fg)_\chi:=U(\fg)\underset{Z(\fg)}\otimes k,$$
where $Z(\fg)\to k$ is the homomorphism $\chi$ corresponding to $\lambda$ via the Harish-Chandra map $Z(\fg)\to \Sym(\ft)$. 

\medskip

It is a theorem of Kostant that in this case the corresponding map
$$U(\fg)_\chi\to \Gamma(X,\on{D}_\lambda)$$
is an isomorphism, i.e., the functor $\on{Loc}$ is fully faithful.

\medskip

The following is the Localization Theorem of \cite{BB1} (amplified by \cite{BB2}):

\begin{thm} \label{t:localization} Consider $\Dmod_\lambda(X)$ as equipped with the \emph{left} D-module t-structure.

\smallskip

\noindent{\em(a)}
Let $\lambda$ be such that $\check\alpha(\lambda)\neq 0$ for any coroot $\check\alpha$ of $G$
(we call such $\lambda$ ``regular"). Then the functors
$\Gamma$ and $\on{Loc}$ are mutually inverse equivalences. 

\smallskip

\noindent{\em(b)}
Let $\lambda$ be such that $\check\alpha(\lambda)\notin \BZ^{<0}$ for any coroot $\check\alpha$ of $G$
(we call such $\lambda$ ``dominant"). Then
The functor $\Gamma$ is t-exact.

\smallskip

\noindent{\em(c)} 
Let $\lambda$ be regular and dominant. Then
the functors $\Gamma$ and $\on{Loc}$ are mutually inverse equivalences, compatible with the t-structures.

\end{thm} 

\sssec{A warning}

When $\chi$ is \emph{irregular}, the algebra $U(\fg)_\chi$ has an infinite cohomological dimension. This implies, in particular,
that the natural t-structure on $\fg\mod_\chi$ does \emph{not} descend to one on $\fg\mod_\chi^c$. 

\medskip

Let $\fg\mod_\chi^{\on{f.g.}}\subset \fg\mod_\chi$ be the full subcategory consisting of objects with finitely
many non-vanishing cohomology groups, and with each cohomology finitely generated as a $U(\fg)$-module.
We have
$$\fg\mod_\chi^c\subset \fg\mod_\chi^{\on{f.g.}},$$
but this inclusion is \emph{not} an equality (the latter would be equivalent to $U(\fg)_\chi$ having a 
finite cohomological dimension). 

\medskip

The functor $\Gamma(X,-)$ sends compact objects in $\Dmod_\lambda(X)$ to $\fg\mod_\chi^{\on{f.g.}}$,
but not necessarily to $\fg\mod_\chi^c$.

\medskip

A related fact is that in this case, the functor $\on{Loc}$ has an unbounded cohomological amplitude on the left
(it is right t-exact, being the left adjoint of a left t-exact functor $\Gamma(X,-)$). 

\sssec{}  

In what follows we will need the following observation:

\begin{prop} \label{p:good Av!}
Let $H_1\subset H_2$ be a pair of subgroups of $G$, and let $\CM$ be an object of $\fg\mod_\chi^{H_1}$.
Assume that the functor $\Av^{H_2/H_1}_!$ is defined on $\Loc(\CM)$. Then $\Av^{H_2/H_1}_!$ is defined 
on $\CM$ and we have
$$\Loc(\Av^{H_2/H_1}_!(\CM))\simeq \Av^{H_2/H_1}_!(\Loc(\CM)) \text{ and }
\Av^{H_2/H_1}_!(\CM)\simeq \Gamma\circ \Av^{H_2/H_1}_!\circ \Loc(\CM).$$
\end{prop}

We will prove the following abstract version of \propref{p:good Av!}. Let 
$$
\CD
\CC'_2 @ >{i_2}>> \CC_2 \\
@V{G'}VV  @VV{G}V  \\
\CC'_1 @ >{i_1}>> \CC_1
\endCD
$$
be a commutative diagram of categories, and assume that the diagram
$$
\CD
\CC'_2 @ >{i_2}>> \CC_2 \\
@A{F'}AA  @AA{F}A  \\
\CC'_1 @ >{i_1}>> \CC_1, 
\endCD
$$
where $F$ (resp., $F'$) is the left adjoint of $G$ (resp., $G'$), also commutes. 

\medskip

Let $\bc_1\in \CC_1$ be an object, and assume that the partially defined left adjoint $(i_2)^L$ is defined on $F(\bc_1)$.

\begin{prop}
Assume that the functors $F$ and $F'$ are fully faithful. Then the partially defined left adjoint $(i_1)^L$ is defined on $\bc_1$
and $F'((i_1)^L(\bc_1))\simeq (i_2)^L(F(\bc_1))$.
\end{prop}

\begin{proof}

Denote $\bc_2:=F(\bc_1)$ and 
$\bc'_2:=(i_2)^L(\bc_2)$. We claim that the map
\begin{equation} \label{e:bad map}
F'\circ G'(\bc'_2)\to \bc'_2
\end{equation}
is an isomorphism. (If this is this case, then  it is easy to see that the object $\bc'_1:=G(\bc'_2)$ is the value of $(i_1)^L$ on $\bc_1$). 

\medskip

Let $\wt\bc'_2$ denote the cone of \eqref{e:bad map}. Then $G'(\wt\bc'_2)=0$. We claim that this implies that $\wt\bc'_2=0$,
which is equivalent to the map $\bc'_2\to \wt\bc'_2$ being zero. 

\medskip

Indeed, for \emph{any}  $\wt\bc'_2\in \CC'_2$, we have
$$\CHom_{\CC'_2}(\bc'_2,\wt\bc'_2)\simeq 
\CHom_{\CC_2}(\bc_2,i_2(\wt\bc'_2))\simeq \CHom_{\CC_1}(\bc_1,G\circ i_2(\wt\bc'_2))\simeq
\CHom_{\CC_1}(\bc_1,i_1\circ G'(\wt\bc'_2)).$$
 
Hence, if $G'(\wt\bc'_2)=0$, then $\CHom_{\CC'_2}(\bc'_2,\wt\bc'_2)=0$, as desired.

\end{proof} 

\ssec{Translation functors} \label{ss:trans}

In this subsection we take $G$ to be a reductive group and $X$ its flag variety. 

\sssec{}

Let $\lambda$ be a dominant weight and let $\mu$ be a dominant integral weight. Set $\lambda'=\lambda+\mu$, and let
$\chi$ and $\chi'$ be the corresponding characters of $Z(\fg)$. 

\medskip

Let $V^\mu$ be the irreducible $G$-module with highest
weight $\lambda$.

\sssec{}  \label{sss:bimodule}

Let us view $U(\fg)_{\chi'}\otimes  (V^\mu)^*$ as a $(U(\fg),U(\fg)_{\chi'})$-bimodule, where $U(\fg)_{\chi'}$
acts on the $U(\fg)_{\chi'}$-factor in $U(\fg)_{\chi'}\otimes  (V^\mu)^*$ by right multiplication and $U(\fg)$
acts diagonally, with the action on the $U(\fg)_{\chi'}$-factor being that by left multiplication. 

\medskip

It is easy to show that when viewed as a $U(\fg)$-module, $U(\fg)_{\chi'}\otimes  (V^\mu)^*$ splits as a direct sum 
$$(U(\fg)_{\chi'}\otimes (V^\mu)^*)_{\chi} \oplus (U(\fg)_{\chi'}\otimes (V^\mu)^*)_{\neq \chi},$$
where the set-theoretic support of $(U(\fg)_{\chi'}\otimes (V^\mu)^*)_{\chi}$ over $\Spec(Z(\fg))$ is $\{\chi\}$ and
the set-theoretic support of $(U(\fg)_{\chi'}\otimes (V^\mu)^*)_{\neq \chi}$ over $\Spec(Z(\fg))$ is finite and disjoint from $\chi$. 

\medskip

This decomposition automatically respects the right $U(\fg)_{\chi'}$-action. 

\medskip

The key observation is that the assumptions on $\lambda$ and $\lambda'$ imply that $(U(\fg)_{\chi'}\otimes (V^\mu)^*)_{\chi}$
is \emph{scheme-theoretically} supported at $\chi\in \Spec(Z(\fg))$, i.e., $(U(\fg)_{\chi'}\otimes (V^\mu)^*)_{\chi}$
is well-defined as an object of $(\fg\mod_\chi)^\heartsuit$. 

\sssec{}

We define the functor 
$$T_{\chi'\to \chi}:\fg\mod_{\chi'}\to\fg\mod_\chi$$
to be given by
$$\CM\mapsto (U(\fg)_{\chi'}\otimes (V^\mu)^*)_{\chi} \underset{U(\fg)_{\chi'}}\otimes \CM.$$

\medskip

This functor is t-exact, since $(U(\fg)_{\chi'}\otimes (V^\mu)^*)_{\chi}$ is projective
as a right $U(\fg)_{\chi'}$-module (being a direct summand of such). 
At the level of abelian categories it is described as follows: for $\CM\in \fg\mod_{\chi'}$, the tensor
product $\CM\otimes (V^\mu)^*$ splits as a direct sum
$$(\CM\otimes (V^\mu)^*)_{\chi}\oplus (\CM\otimes (V^\mu)^*)_{\neq \chi},$$
where $(\CM\otimes (V^\mu)^*)_{\chi}\in \fg\mod_\chi$ and the 
set-theoretic support of $(\CM\otimes (V^\mu)^*)_{\neq \chi}$ over $\Spec(Z(\fg))$ is finite and disjoint from $\chi$. Then
\begin{equation} \label{e:naive translation}
T_{\chi'\to \chi}(\CM)=(\CM\otimes (V^\mu)^*)_{\chi}.
\end{equation}

\begin{rem}
At the level of derived categories,  we had to define the functor $T_{\chi'\to \chi}$ using the bimodule
$(U(\fg)_{\chi'}\otimes (V^\mu)^*)_{\chi}$, rather than by the formula \eqref{e:naive translation}, because
for an object of $\fg\mod$, to belong to $\fg\mod_\chi$ is not a property, but extra structure.
\end{rem} 

\sssec{}

The basic property of the translation functor $T_{\chi\to \chi'}$ is that it makes the following diagram commute:
$$
\CD
\Dmod_{\lambda'}(X)  @>{-\otimes \CO(-\mu)}>>  \Dmod_{\lambda}(X)  \\
@V{\Gamma}VV  @VV{\Gamma}V  \\
\fg\mod_{\chi'}   @>{T_{\chi'\to \chi}}>> \fg\mod_\chi,
\endCD
$$
where the $G$-equivariant line bundle $\CO(\mu)$ on $X$ is normalized so that it is ample
and $\Gamma(X,\CO(\mu))\simeq V^\mu$.

\sssec{}  \label{sss:bad trans}

We let  
$$T_{\chi\to \chi'}:\fg\mod_{\chi}\to\fg\mod_{\chi'}$$
be the left adjoint functor of $T_{\chi'\to \chi}$. 

\medskip

Tautologically, it makes the following diagram commute:
$$
\CD
\Dmod_{\lambda}(X)  @>{-\otimes \CO(\mu)}>>  \Dmod_{\lambda'}(X)  \\
@A{\on{Loc}}AA  @AA{\on{Loc}}A  \\
\fg\mod_{\chi}   @>{T_{\chi\to \chi'}}>> \fg\mod_{\chi'}.
\endCD
$$

\medskip

Consider the object
$$T_{\chi\to \chi'}(U(\fg)_\chi)\in \fg\mod_{\chi'}.$$

\medskip

The functor $T_{\chi\to \chi'}$ is given by
$$\CM\mapsto T_{\chi\to \chi'}(U(\fg)_\chi)\underset{U(\fg)_\chi}\otimes \CM,$$
where we regard $T_{\chi\to \chi'}(U(\fg)_\chi)$ as a $(U(\fg)_{\chi'},U(\fg)_\chi)$-bimodule. 

\medskip

For what follows, we will need to describe the above object $T_{\chi\to \chi'}(U(\fg)_\chi)$ more explicitly. 

\sssec{}  \label{sss:trans g}

Consider the tensor product
$$U(\fg)_\chi\otimes V^\mu$$
as a $(U(\fg),U(\fg)_\chi)$-bimodule as in \secref{sss:bimodule}, and consider the corresponding decomposition
$$(U(\fg)_\chi\otimes V^\mu)_{\chi'} \oplus (U(\fg)_\chi\otimes V^\mu)_{\neq \chi'}$$
according to set-theoretic support over $\Spec(Z(\fg))$.

\medskip

It is no longer true that $(U(\fg)_\chi\otimes V^\mu)_{\chi'}$ is scheme-theoretically supported at $\chi'\in \Spec(Z(\fg))$.
Let $(U(\fg)_\chi\otimes V^\mu)'_{\chi'}$ be the maximal quotient of $(U(\fg)_\chi\otimes V^\mu)_{\chi'}$ with scheme-theoretic
support at $\chi'$, i.e., the maximal quotient on which $U(\fg)$ acts via $U(\fg)_{\chi'}$. 

\begin{prop}
We have a canonical isomorphism of $(U(\fg)_{\chi'},U(\fg)_\chi)$-bimodules
$$T_{\chi\to \chi'}(U(\fg)_\chi)\simeq (U(\fg)_\chi\otimes V^\mu)'_{\chi'};$$
in particular, $T_{\chi\to \chi'}(U(\fg)_\chi)$ lies in $\fg\mod_{\chi'}^\heartsuit$. 
\end{prop}

\begin{proof} 

Since $T_{\chi\to \chi'}$ is the left adjoint of a t-exact functor and since $U(\fg)_\chi\in \fg\mod_\chi^\heartsuit$ 
is projective, we obtain that $T_{\chi\to \chi'}(U(\fg)_\chi)$ is a projective object in 
$\fg\mod_{\chi'}^\heartsuit$. 

\medskip

The functor
$$\CM\mapsto H^0\left(T_{\chi\to \chi'}(U(\fg)_\chi)\underset{U(\fg)_\chi)}\otimes \CM\right), \quad 
\fg\mod_\chi^\heartsuit\to \fg\mod_{\chi'}^\heartsuit$$
provides a left adjoint to the functor
$$T_{\chi'\to \chi}: \fg\mod_{\chi'}^\heartsuit\to \fg\mod_\chi^\heartsuit.$$

However, it is easy to see that the above left adjoint is given by
$$\CM\mapsto (\CM\otimes V^\mu)'_{\chi'},$$
where:

\begin{itemize}

\item $\CM\otimes V^\mu\simeq (\CM\otimes V^\mu)_{\chi'}\oplus (\CM\otimes V^\mu)_{\neq \chi'}$
is the decomposition of  $\CM\otimes V^\mu$ according to set-theoretic support over $\Spec(Z(\fg))$;

\item $(\CM\otimes V^\mu)'_{\chi'}$ is the maximal quotient of  $(\CM\otimes V^\mu)_{\chi'}$ 
with scheme-theoretic support at $\chi'\in \Spec(Z(\fg))$.

\end{itemize}

Hence, we obtain
$$T_{\chi\to \chi'}(U(\fg)_\chi)\simeq H^0\left(T_{\chi\to \chi'}(U(\fg)_\chi)\right)\simeq (U(\fg)_\chi\otimes V^\mu)'_{\chi'},$$
as desired.

\end{proof}

From here, and using Noetherianness, we obtain:

\begin{cor}  \label{c:bad trans}
The object $T_{\chi\to \chi'}(U(\fg)_\chi)$, viewed as a $(U(\fg),U(\fg)_\chi)$-bimodule admits a
resolution with terms of the form $U(\fg)_\chi \otimes V$, where $V$ is a finite-dimensional representation
of $G$.
\end{cor} 

\ssec{The long intertwining operator}  \label{ss:longest intertwiner}

\sssec{}

In this subsection we take $G$ to be a reductive group and $X$ its flag variety. For a given $\lambda\in \ft^*$, we consider the
corresponding categories $\Dmod_\lambda(X)^N$ and $\Dmod_\lambda(X)^{N^-}$.

\medskip

The goal of this subsection is to supply a (possibly new) proof of the following well-known statement (see also \cite[Theorem 5.2]{CY} for a similar statement in a more general setting of the Matsuki correspondence):

\begin{prop} \label{p:Ups X} 
The partially defined functor $\Av^{N^-}_!$ \emph{is} defined on the essential image of $\oblv_N$, and the composition $\Av^{N^-}_!\circ \oblv_N$
provides an inverse to the functor
$$\Upsilon:=\Av^N_*\circ \oblv_{N^-}, \quad \Dmod_\lambda(X)^{N^-}\to \Dmod_\lambda(X)^N.$$
\end{prop}

As a consequence we will now deduce:

\begin{prop} \label{p:Ups g} 
Consider the action of $G$ on the category $\fg\mod_\chi$ for a given central character $\chi$. Then
the partially defined functor $\Av^{N^-}_!$ \emph{is} defined on the essential image of 
$$\oblv_N:\fg\mod_\chi^N\to \fg\mod_\chi,$$
and the composition $\Av^{N^-}_!\circ \oblv_N$ provides an inverse to the functor
$$\Upsilon:=\Av^N_*\circ \oblv_{N^-}, \quad \fg\mod_\chi^{N^-}\to \fg\mod_\chi^N.$$
\end{prop}

\begin{proof}[Proof of \propref{p:Ups g}]

Choose a twisting $\lambda$ that gives rise to $\chi$ via the Harish-Chandra homomorphism. The fact that $\Av^{N^-}_!$
is defined on $\fg\mod_\chi^N$ follows from the corresponding fact for $\Dmod_\lambda(X)^N$ using \propref{p:good Av!}.
Moreover, for $\CM\in \fg\mod_\chi^N$, we have:
$$\on{Loc}\circ \Av^{N^-}_!\circ \oblv_N(\CM)\simeq \Av^{N^-}_!\circ \on{Loc}\circ \oblv_N(\CM)$$ 
%\text{ and } \Av^{N^-}_!\circ \oblv_N(\CM)\simeq \Gamma \circ \Av^{N^-}_! \circ \oblv_N\circ \on{Loc}(\CM).$$

\medskip

Further, we have
\begin{multline*} 
(\Av^N_*\circ \oblv_{N^-})\circ (\Av^{N^-}_!\circ \oblv_N) \simeq \Gamma\circ \on{Loc} \circ 
(\Av^N_*\circ \oblv_{N^-})\circ (\Av^{N^-}_!\circ \oblv_N) \simeq \\
\Gamma\circ (\Av^N_*\circ \oblv_{N^-})\circ (\on{Loc} \circ  \Av^{N^-}_! \circ \oblv_N) 
\simeq \Gamma\circ (\Av^N_*\circ \oblv_{N^-})\circ (\Av^{N^-}_!\circ \on{Loc} \circ \oblv_N) \simeq \\
\simeq \Gamma\circ (\Av^N_*\circ \oblv_{N^-})\circ (\Av^{N^-}_!\circ \oblv_N)\circ \on{Loc}
\simeq \Gamma\circ \on{Loc}\simeq \on{Id},
\end{multline*} 
where the fifth isomorphism
follows from \propref{p:Ups X}.

\medskip

Similarly,
\begin{multline*} 
(\Av^{N^-}_!\circ \oblv_N) \circ (\Av^N_*\circ \oblv_{N^-}) \simeq 
\Gamma\circ \on{Loc}\circ (\Av^{N^-}_!\circ \oblv_N)  \circ  (\Av^N_*\circ \oblv_{N^-}) \simeq \\ 
\simeq \Gamma\circ  \Av^{N^-}_! \circ \on{Loc}\circ \oblv_N  \circ  (\Av^N_*\circ \oblv_{N^-}) \simeq 
\Gamma\circ  (\Av^{N^-}_! \circ  \oblv_N)  \circ  (\Av^N_*\circ \oblv_{N^-})\circ \on{Loc} \simeq \\
\simeq \Gamma\circ \on{Loc}\simeq \on{Id}.
\end{multline*} 

\end{proof}

\sssec{}

The rest of this subsection is devoted to the proof of \propref{p:Ups X}. The key idea is to use the theorem of T.~Braden's 
(see, e.g., \cite{DrGa2}).

\medskip

Choose a regular dominant coweight $\BG_m\to T$, and consider the resulting $\BG_m$-action on $X$. Note 
that the orbits of $N$ (resp., $N^-$) are $\BG_m$-invariant. In particular, every $N$- or $N^-$-equivariant D-module on $X$
is $\BG_m$-monodromic. 

\medskip

Moreover, the disjoint union of $N$-orbits, to be denoted $X^+$ (resp., the disjoint union of $N^-$-orbits, to be denoted $X^-$)
is the attracting (resp., repelling) locus for the above action. Denote $X^0=X^{\BG_m}$ and denote by
$$i^+:X^0\rightleftarrows X^+:q^+, \quad p^+:X^+\to X$$
and 
$$i^-:X^0\rightleftarrows X^-:q^-, \quad p^-:X^-\to X$$
the corresponding maps.

\medskip

We will consider the functor of \emph{hyperbolic restriction} from the full subcategory
$$\Dmod_\lambda(X)^{\on{mon}}\subset \Dmod_\lambda(X)$$
consisting of $\BG_m$-monodromic objects to $\Dmod_\lambda(X^0)$. Braden's theorem insures that this functor
is well-defined:
$$(i^-)^!\circ (p^-)^*\simeq (q^-)_!\circ (p^-)^*=:\sH^-\underset{\sim}{\overset{\Phi}\leftarrow} \sH^+:=(q^+)_*\circ (p^+)^!\simeq (i^+)^*\circ (p^+)^!,$$
where the isomorphisms on the sides are known as the \emph{contraction principle}. 

\sssec{}

The fact that $\Av^{N^-}_!$ is defined on the essential image of $\oblv_N$ follows from the holonomicity.
Here is, however, an alternative argument:

\medskip

We claim that $\Av^{N^-}_!$ is defined on all of $\Dmod_\lambda(X)^{\on{mon}}$. Indeed, for an
object $\CF\in \Dmod_\lambda(X)$, the object $\Av^{N^-}_!(\CF)$ is defined if 
$$(p^-)^*\circ  \Av^{N^-}_!(\CF)\in \Dmod_\lambda(X^-)^{N^-}$$
is defined, and that is if and only if
$$(i^-)^!\circ \oblv_{N^-}\circ (p^-)^* \circ \Av^{N^-}_!(\CF)\in \Dmod_\lambda(X^0)$$
is defined. 

\medskip

We rewrite
\begin{multline} \label{e:hyp calc}
(i^-)^!\circ \oblv_{N^-}\circ (p^-)^*\circ \Av^{N^-}_!(\CF)\simeq 
(q^-)_!\circ \oblv_{N^-}\circ (p^-)^*\circ \Av^{N^-}_!(\CF)\simeq  \\
\simeq (q^-)_!\circ \oblv_{N^-}\circ  \Av^{N^-}_! \circ (p^-)^*(\CF)\simeq 
(q^-)_!\circ (p^-)^*(\CF)\simeq \sH^-(\CF),
\end{multline}
where the first isomorphism is given by the contraction principle, and 
the third isomorphism is due to the fact that the projection $q^-$ is $N^-$-invariant. 

\sssec{}

From \eqref{e:hyp calc} we obtain a canonical isomorphism
\begin{equation} \label{e:main calc 1}
\sH^-\circ \oblv_{N^-} \circ \Av^{N^-}_!(\CF) \simeq \sH^-(\CF), \quad \CF\in \Dmod_\lambda(X)^{\on{mon}}.
\end{equation}

\medskip

Similarly, we have:
\begin{multline*} 
(i^+)^*\circ \oblv_N \circ (p^+)^!\circ \Av^N_*(\CF)\simeq  
(q^+)_*\circ \oblv_N\circ (p^+)^!\circ \Av^N_*(\CF)\simeq  \\
\simeq (q^+)_*\circ \oblv_N\circ \Av^N_* \circ (p^+)^!(\CF) \simeq 
(q^+)_*\circ (p^+)^!(\CF)\simeq \sH^+(\CF),
\end{multline*}
i.e.,
\begin{equation} \label{e:main calc 2}
\sH^+\circ \oblv_N \circ \Av^N_*(\CF) \simeq \sH^+(\CF), \quad \CF\in \Dmod_\lambda(X)^{\on{mon}}.
\end{equation}

\sssec{}

A priori, the functor $\Av^{N^-}_!\circ \oblv_N$ is the left adjoint of the functor $\Av^N_*\circ \oblv_{N^-}$.
Let us prove that the unit of the adjunction 
\begin{equation} \label{e:unit ups}
\CF\to (\Av^N_*\circ \oblv_{N^-})\circ (\Av^{N^-}_!\circ \oblv_N)(\CF)
\end{equation}
is an isomorphism.

\medskip

To show that \eqref{e:unit ups} is an isomorphism, it is enough to show that it induces an isomorphism after applying the functor 
$\sH^+\circ \oblv_N$. The resulting map is
\begin{multline*} 
\sH^+\circ \oblv_N(\CF)\to 
(\sH^+\circ \oblv_N) \circ (\Av^N_*\circ \oblv_{N^-})\circ (\Av^{N^-}_!\circ \oblv_N)(\CF) \overset{\text{\eqref{e:main calc 2}}} \simeq \\
\simeq \sH^+\circ  \oblv_{N^-}\circ (\Av^{N^-}_!\circ \oblv_N)(\CF)\overset{\Phi} \to \sH^-\circ  \oblv_{N^-}\circ (\Av^{N^-}_!\circ \oblv_N)(\CF) 
\overset{\text{\eqref{e:main calc 1}}} \simeq \\
\simeq \sH^-\circ  \oblv_N(\CF).
\end{multline*}

By unwinding the definition of the natural transformation $\Phi:\sH^+\to \sH^-$ (see \cite[Equation (3.5)]{DrGa2}), we obtain that the above map is induced by 
the map
$$\sH^+\circ \oblv_N(\CF)\overset{\Phi}\to \sH^-\circ \oblv_N(\CF),$$
and hence is an isomorphism by Braden's theorem. 

\sssec{}

It remains to show that the functor $\Av^N_*\circ \oblv_{N^-}$ is conservative. But this follows from the fact that its composition
with $\sH^+\circ \oblv_N$ is conservative:
$$\sH^+\circ \oblv_N \circ \Av^N_*\circ \oblv_{N^-}  \overset{\text{\eqref{e:main calc 2}}} \simeq
\sH^+\circ \oblv_{N^-} \overset{\Phi}\simeq \sH^-\circ \oblv_{N^-},$$
while the latter is evidently conservative on $\Dmod_\lambda(X)^{N^-}$.

\section{Casselman-Jacquet functor as averaging} \label{s:J}

In this section we introduce the main character of this paper -- the Casselman-Jacquet functor, and state the main results. 

\medskip

Throughout this section $G$ is a reductive group, $X$ is its flag variety.  We let $N$ (resp., $N^-$) be the unipotent radical of a Borel
subgroup (resp., an opposite Borel) in $G$. 

\ssec{Casselman-Jacquet functor in the abstract setting}  \label{ss:abs J}

\sssec{}

Let $\CC$ be a category acted on by $N$. Consider the (fully faithful) forgetful functor
$$\oblv_N:\CC^N\to \CC.$$

As was mentioned in \secref{sss:averaging}, it admits a right adjoint functor $\Av^N_*$, so that
we have an adjoint pair $(\oblv_N,\Av^N_*)$.

\medskip

We will now consider the (usually, \emph{discontinuous}) right adjoint of $\Av^N_*$, denoted
$$(\Av^N_*)^R:\CC^N\to \CC.$$

We will also consider the monad $(\Av^N_*)^R\circ \Av^N_*$ acting on $\CC$.

\begin{rem}
As we shall see in \secref{ss:ident}, the monad $(\Av^N_*)^R\circ \Av^N_*$ is not
as bizarre as one could initially think: in the case when $\CC$ is the category of modules over
an associative algebra, equipped with a Harish-Chandra structure with respect to $N$,
the functor $(\Av^N_*)^R\circ \Av^N_*$ is that of $\fn$-adic completion. 
\end{rem} 

\sssec{}

Suppose now that the $N$-action on $\CC$ comes by restriction from a $G$-action.
We define the functor
$$J:\CC\to \CC^{N^-}$$ 
to be the composition
$$\Av^{N^-}_*\circ (\Av^N_*)^R\circ \Av^N_*.$$

By construction, the functor $J$ is the right adjoint of the functor
$$\oblv_N \circ \Av^N_* \circ \oblv_{N^-}:\CC^{N^-}\to \CC.$$

From here we obtain:

\begin{prop} \label{p:main}
Let $\CC$ be such that the functor 
$$\Upsilon:=\Av^N_* \circ \oblv_{N^-}:\CC^{N^-}\to \CC^N$$
is an equivalence. Then the functor $J$ identifies with
$$\on{Av}^{N^-}_!\circ \oblv_N\circ \Av^N_*.$$
\end{prop} 

\begin{proof}

If $\Upsilon$ is an equivalence, then the functor 
$\on{Av}^{N^-}_!$ is defined on the essential image of $\oblv_N$, and 
$$\on{Av}^{N^-}_!\circ \oblv_N: \CC^N\to \CC^{N^-},$$
which is the (a priori partially defined) left adjoint of $\Upsilon$, is well-defined and is 
the inverse of $\Upsilon$. 

\medskip

In particular, $\on{Av}^{N^-}_!\circ \oblv_N$ is \emph{right} adjoint to $\Upsilon$, so
$$\Av^{N^-}_*\circ (\Av^N_*)^R\simeq \Upsilon^R\simeq \on{Av}^{N^-}_!\circ \oblv_N.$$

Composing, we obtain
$$J=\Av^{N^-}_*\circ (\Av^N_*)^R\circ \Av^N_*\simeq \on{Av}^{N^-}_!\circ \oblv_N \circ \Av^N_*,$$
as desired.

\end{proof}

\begin{rem}
According to \secref{ss:longest intertwiner}, the functor $\Upsilon$ is an equivalence in the 
following cases of interest:
$$\CC=\Dmod_\lambda(X) \text{ and } \CC=\fg\mod_\chi.$$
\end{rem} 

\ssec{Casselman-Jacquet functor as completion}  \label{ss:completion}

In this subsection $N$ can be any unipotent group. 

\sssec{}

In this subsection we will consider the particular case of $\CC=\fn\mod$, equipped with the natural
$N$-action. Note that in this case $\CC^N\simeq \Rep(N)$.

\medskip

We will see that the endofunctor $(\Av^N_*)^R\circ \Av^N_*$ of $\fn\mod$,
can be described explicitly as the functor of $\fn$-adic completion. 

\sssec{}

Consider the inverse family of $\fn$-bimodules indexed by natural numbers
$$i\mapsto U(\fn)^i:=U(\fn)/\fn^i\cdot U(\fn),$$
and the corresponding family of endofunctors of $\fn\mod$:
$$\CM\mapsto U(\fn)^i\underset{U(\fn)}\otimes \CM=:\CM/\fn^i\cdot \CM.$$

We claim:

\begin{prop} \label{p:completion}
The functor $(\Av^N_*)^R\circ \Av^N_*$ identifies canonically with
$$\CM\mapsto \underset{i}{\on{lim}}\, \CM/\fn^i\cdot \CM.$$
\end{prop}

\sssec{Example}
Let $N=\BG_a$. We identify the category $\fn\mod$ with $k[t]\mod$, and 
$$\Rep(N)\subset \fn\mod$$
with 
$$k[t]\mod_{\{0\}}\subset k[t]\mod,$$
the full subcategory consisting of modules, on whose cohomologies $t$ acts locally nilpotently. 

\medskip

In this case, the assertion of the proposition is well-known: the functor in question is that of
$t$-adic completion.

\sssec{}

The rest of this subsection is devoted to the proof of \propref{p:completion}.

\medskip

Let $\fn\mod^{\on{non-nilp}}\subset \fn\mod$ be the full subcategory equal to the \emph{right orthogonal} of $\Rep(N)$.
The assertion of the proposition amounts to the following two statements:

\medskip

\noindent{(a)} For $\CM\in \fn\mod^{\on{non-nilp}}$, the limit $\underset{i}{\on{lim}}\, \CM/\fn^i\cdot \CM$ is zero.

\medskip

\noindent{(b)} For any $\CM\in \Rep(N)$, the cofiber of the canonical map
$$\CM\to \underset{i}{\on{lim}}\, \CM/\fn^i\cdot \CM$$ belongs to $\fn\mod^{\on{non-nilp}}$. 

\medskip

Let $\triv\in \Rep(N)\subset \fn\mod$ denote the trivial representation. First we notice:

\medskip

\begin{lem}  \label{l:non-nilp}
$\CM\in \fn\mod^{\on{non-nilp}} \,\Leftrightarrow\, \triv\underset{U(\fn)}\otimes \CM=0$.
\end{lem}

\begin{proof} 
The category $\Rep(N)$ is compactly generated by $\triv$. Hence,
$$\CM\in \fn\mod^{\on{non-nilp}} \,\Leftrightarrow\,\CHom_{\fn\mod}(\triv,\CM)=0 \,\Leftrightarrow\, \on{C}^\bullet(\fn,\CM)=0.$$

Since, $\on{C}^\bullet(\fn,-)$ is isomorphic to, up to a cohomological shift, to $\on{C}_\bullet(\fn,-)$, we obtain that
$$\CM\in \fn\mod^{\on{non-nilp}} \,\Leftrightarrow\, \on{C}_\bullet(\fn,\CM)=0 \,\Leftrightarrow\,
\triv\underset{U(\fn)}\otimes \CM=0.$$
\end{proof}

\begin{proof}[Proof of \em{(a)}]
Since each $U(\fn)^i$ admits a filtration with subquotients isomorphic to $\triv$, from the Lemma, we obtain that if 
$\CM\in \fn\mod^{\on{non-nilp}}$, then all the terms in 
$\underset{i}{\on{lim}}\, \CM/\fn^i\cdot \CM$ are zero.

\end{proof} 

\begin{proof}[Proof of \em{(b)}]

We need to show that for any $\CM\in \Rep(N)$, the map
$$\triv\underset{U(\fn)}\otimes \CM \to \triv\underset{U(\fn)}\otimes \left(\underset{i}{\on{lim}}\, U(\fn)^i \underset{U(\fn)}\otimes \CM\right)$$
is an isomorphism. We know that the functor 
$$\triv\underset{U(\fn)}\otimes-\simeq \on{C}_\bullet(\fn,-)$$ 
commutes with limits (again, because it is isomorphic up to a shift to $\on{C}^\bullet(\fn,-)$), so we need to show that the map
$$\triv\underset{U(\fn)}\otimes \CM \to \underset{i}{\on{lim}}\, \left(\triv\underset{U(\fn)}\otimes U(\fn)^i \underset{U(\fn)}\otimes \CM\right)$$
is an isomorphism. In other words, we need to show that 
$$\underset{i}{\on{lim}}\, \left(\on{coFib}(\triv\to \triv\underset{U(\fn)}\otimes U(\fn)^i) \underset{U(\fn)}\otimes \CM\right)$$
is zero. 

\medskip

Since $\Rep(N)$ is \emph{co-generated} by $R_N$, it suffices to show that 
$$\underset{i}{\on{lim}}\, \left(\on{coFib}(\triv\to \triv\underset{U(\fn)}\otimes U(\fn)^i) \underset{U(\fn)}\otimes R_N\right)$$
is zero. 

\medskip

Now, since $R_N$ admits a \emph{finite} resolution with terms of the form $U(\fn)\otimes V$ (where $V\in \Vect$), it suffices to
show that 
$$\underset{i}{\on{lim}}\, \left(\on{coFib}(\triv\to \triv\underset{U(\fn)}\otimes U(\fn)^i) \underset{U(\fn)}\otimes U(\fn)\otimes V\right)
\simeq \underset{i}{\on{lim}}\, \left(\on{coFib}(k\to \triv\underset{U(\fn)}\otimes U(\fn)^i)\otimes V\right)$$
is zero. 

\medskip

We will show that the inverse system (in $\Vect$)
$$i\mapsto \on{coFib}(k\to \triv\underset{U(\fn)}\otimes U(\fn)^i)$$ is \emph{null}, i.e., for every $i$ there exists $i'>i$,
such that the map
$$\on{coFib}(k\to \triv\underset{U(\fn)}\otimes U(\fn)^{i'})\to \on{coFib}(k\to \triv\underset{U(\fn)}\otimes U(\fn)^i)$$
is zero. 

\medskip

The objects of $\Vect$ involved are compact (i.e., have finitely many cohomologies and each cohomology is finite-dimensional).
Hence, we can dualize, and the assertion becomes equivalent to the fact that the direct system
$$i\mapsto \on{Fib}\left(\on{C}^\bullet(\fn,(U(\fn)^i)^*)\to k\right)$$
is null. Again, by compactness, the latter is equivalent to the fact that 
$$\underset{i}{\on{colim}}\, \on{Fib}\left(\on{C}^\bullet(\fn,(U(\fn)^i)^*)\to k\right)=0,$$
i.e., that the map
$$\underset{i}{\on{colim}}\, \on{C}^\bullet(\fn,(U(\fn)^i)^*)\to k$$
is an isomorphism. Since $\on{C}^\bullet(\fn,-)$ commutes with colimits (being isomorphic up to a shift to $\on{C}_\bullet(\fn,-)$ ), 
this is equivalent to the map
\begin{equation} \label{e:ev}
\on{C}^\bullet\left(\fn, (\underset{i}{\on{colim}}\,( U(\fn)^i)^*\right)\to k
\end{equation}
being an isomorphism. 

\medskip

We now notice that the pairing
$$f,u \mapsto u(f)(1), \quad f\in R_N,\,\,u\in U(\fn)$$
defines an isomorphism 
$$R_N\simeq \underset{i}{\on{colim}}\, (U(\fn)^i)^*$$
as $\fn$-modules. 

\medskip

Under this isomorphism, the above map \eqref{e:ev} identifies with
$$\on{C}^\bullet(\fn, R_N)\to R_N \overset{f\mapsto f(1)}\longrightarrow k,$$
which is evidently an isomorphism.

\end{proof}

\qed[\propref{p:completion}]

\ssec{The Casselman-Jacquet functor for $A$-modules}  \label{ss:ident}

\sssec{}

Let $A$ be an associative algebra equipped with a Harish-Chandra structure with respect to $N$ (see \secref{sss:H-Ch algebra}
for what this means). In particular,
$A\mod$ is acted on by $N$, and we have the restriction functor 
$$A\mod\to \fn\mod,$$
and its left adjoint $\fn\mod\to A\mod$, both compatible with $N$-actions.

\medskip

We have commutative diagrams
$$
\CD
A\mod^N @>{\oblv_N}>> A\mod @>{\Av^N_*}>> A\mod^N  \\
@VVV   @VVV   @VVV  \\
\Rep(N) @>{\oblv_N}>> \fn\mod @>{\Av^N_*}>> \Rep(N)
\endCD
$$
and
$$
\CD
A\mod^N @>{\oblv_N}>> A\mod @>{\Av^N_*}>> A\mod^N  \\
@AAA   @AAA   @AAA  \\
\Rep(N) @>{\oblv_N}>> \fn\mod @>{\Av^N_*}>> \Rep(N).
\endCD
$$

\medskip

By passing to right adjoints in the second diagram, we obtain that the following diagram commutes as well:
\begin{equation} \label{e:other diag com}
\CD
A\mod @>{\Av^N_*}>> A\mod^N @>{(\Av^N_*)^R}>> A\mod  \\
@VVV  @VVV  @VVV  \\
\fn\mod @>{\Av^N_*}>> \Rep(N) @>{(\Av^N_*)^R}>> \fn \mod.
\endCD
\end{equation}

\sssec{}

From the commutation of \eqref{e:other diag com} and \propref{p:completion} we obtain:

\begin{cor}
The endofunctor $(\Av^N_*)^R\circ \Av^N_*$ on $A\mod$ is right t-exact.
\end{cor} 

\begin{rem}
By unwinding the constructions one can show that for $\CM\in A\mod^\heartsuit$, the action of $A$
on 
$$H^0((\Av^N_*)^R\circ \Av^N_*(\CM))\simeq \underset{k}{\on{lim}}\, H^0(\CM/\fn^k\cdot \CM)$$
is given by the following formula:

\medskip

For an element $a\in A$, let $k$ be an integer so that 
$$\on{ad}_{\xi_1}\circ ...\circ \on{ad}_{\xi_k}(a)=0 \quad \forall \xi_1,...,\xi_k\in \fn.$$

Then the action of $a$ is well-defined as a map
$$H^0(\CM/\fn^{k'+k}\cdot \CM)\to H^0(\CM/\fn^{k'}\cdot \CM),$$
thereby giving a map on the inverse limit.
\end{rem}

\sssec{}  \label{sss:n fin}

Let $A$ be left-Noetherian. Let 
\begin{equation} \label{e:n fg}
A\mod^{\fn\on{-f.g.}}\subset A\mod
\end{equation}
be the full subcategory consisting of modules that map to compact (i.e., finitely generated) objects
under the forgetful functor $A\mod \to \fn\mod$. Note that $A\mod^{\fn\on{-f.g.}}$ is \emph{not necessarily}
contained in $A\mod^c$ (unless $A$ has a finite cohomological dimension). 

\medskip 

Let 
$$\Ind(A\mod^{\fn\on{-f.g.}})$$
be the ind-completion of $A\mod^{\fn\on{-f.g.}}$. Ind-extending the tautological embedding
$$A\mod^{\fn\on{-f.g.}}\hookrightarrow A\mod,$$
we obtain a functor 
\begin{equation} \label{e:Ind n fg}
\Ind(A\mod^{\fn\on{-f.g.}}) \to A\mod.
\end{equation}

Ind-extending the t-structure on $A\mod^{\fn\on{-f.g.}}$, we obtain one on $\Ind(A\mod^{\fn\on{-f.g.}})$.
The functor \eqref{e:Ind n fg} is t-exact, but not in general fully faithful. However, as in \cite[Proposition 2.3.3]{Ga4},
one shows that the functors
\begin{equation} \label{e:Ind n fg trunc}
\Ind(A\mod^{\fn\on{-f.g.}})^{\geq -n}\to A\mod^{\geq -n}
\end{equation}
are fully faithful for every $n$.

\medskip

Let $\Ind^\wedge(A\mod^{\fn\on{-f.g.}})$ denote the left-completion of $\Ind(A\mod^{\fn\on{-f.g.}})$
in its t-structure. Since $A\mod$ is left-complete in its t-structure, the functor \eqref{e:Ind n fg} 
extends to a functor
\begin{equation} \label{e:Ind n fg compl}
\Ind^\wedge(A\mod^{\fn\on{-f.g.}})\to A\mod.
\end{equation}

The functor \eqref{e:Ind n fg compl} is fully faithful, since the functors \eqref{e:Ind n fg trunc} have this
property. Its essential image consists of objects of $A\mod$, whose cohomologies are filtered colimits of
objects from
$$A\mod^{\heartsuit,\fn\on{-f.g.}}=A\mod^\heartsuit \cap A\mod^{\fn\on{-f.g.}}.$$

\begin{rem} \label{r:ind n fin}
Suppose that $A$ has a finite cohomological dimension, in which case the functor \eqref{e:Ind n fg}
is fully faithful, and hence so is the functor
$$\Ind(A\mod^{\fn\on{-f.g.}})\to \Ind^\wedge(A\mod^{\fn\on{-f.g.}}).$$
However, it is not clear to the authors whether the latter is an equivalence.
\end{rem} 

\sssec{}

Assume now that $A$ has a Harish-Chandra structure with respect to $G$. We claim:

\begin{prop} \label{p:J left t exact} \hfill

\smallskip

\noindent{\em(a)} The functor $J:A\mod\to A\mod^{N^-}$
is left t-exact when restricted to $A\mod^{\fn\on{-f.g.}}$. 

\smallskip

\noindent{\em(b)} \,Assume that the corresponding functor $\Upsilon$ is an equivalence. 
Then the functor $J$ is left t-exact when restricted to $\Ind^\wedge(A\mod^{\fn\on{-f.g.}})$. 
\end{prop} 

\begin{proof} 

First off, point (b) reduces to point (a) as follows: if $\Upsilon$ is an equivalence, then
by \propref{p:main}, the functor $J$ is commutes with colimits. Then we use \lemref{l:Av and t}. 

\medskip

We now prove point (a).

\medskip

The functor $\Av^{N^-}_*$, being the right adjoint of the t-exact functor
$\oblv_{N^-}$, is left t-exact (on all of $A\mod$). Hence, it suffices to show
that the functor $(\Av^N_*)^R\circ \Av^N_*$ is left t-exact when restricted to
$A\mod^{\fn\on{-f.g.}}$. We will show that it is in fact t-exact.

\medskip

Using the commutation of \eqref{e:other diag com}, it suffices to show that the functor
$(\Av^N_*)^R\circ \Av^N_*$ is t-exact when restricted to $\fn\mod^{\on{f.g.}}$. Using 
\propref{p:completion}, it suffices to show that the functor
$$\CM\mapsto \underset{k}{\on{lim}}\, \CM/\fn^k\cdot \CM$$
sends $\CM\in \fn\mod^{\on{f.g.},\heartsuit}$ to an object in $\Vect^\heartsuit$. However, the latter
is known: this is Casselman's generalization of the Artin-Rees lemma.

\end{proof}

\sssec{}

Note that the functor
$$\CM\mapsto H^0(\Av^{N^-}_*(\CM)), \quad A\mod^\heartsuit\to  (A\mod^{N^-})^\heartsuit$$
is that of sending $\CM$ to its submodule consisting of elements that are locally nilpotent with
respect to the action of $\fn^-$. 

\medskip

Hence, from \propref{p:J left t exact} we obtain:

\begin{cor} \hfill

\smallskip

\noindent{\em(a)} Under the assumptions of \propref{p:J left t exact}(a), the functor
$$\CM\mapsto H^0(J(\CM)), \quad (A\mod^{\fn\on{-f.g.}})^\heartsuit\to  (A\mod^{N^-})^\heartsuit$$
sends $\CM$ to the submodule of 
$\underset{k}{\on{lim}}\, \CM/\fn^k\cdot \CM$
consisting of elements that are locally nilpotent with
respect to the action of $\fn^-$. 

\smallskip

\noindent{\em(b)} Under the assumptions of \propref{p:J left t exact}(b), ditto for
the category $(\Ind^\wedge(A\mod^{\fn\on{-f.g.}}))^\heartsuit$.

\end{cor}

\ssec{The Casselman-Jacquet functor for $\fg$-modules}

\sssec{}

The goal of this subsection is to prove the following:

\begin{thm} \label{t:J self-dual on g}
Consider $G$ acting on the category $\fg\mod_\chi$. \hfill

\smallskip

\noindent{\em(a)} 
There is a canonically defined natural transformation of functors
\begin{equation} \label{e:self-duality}
\Av^{N^-}_*\circ \oblv_{N}\circ  \Av^N_!\to 
\Av^{N^-}_!\circ\oblv_{N}\circ  \Av^N_*\simeq J,
\end{equation}
where the LHS is a partially defined functor. 

\smallskip

\noindent{\em(b)} The functor $\Av^N_!$ is defined and the
above natural transformation is an isomorphism, when evaluated on
$\Ind^\wedge(\fg\mod_\chi^{\fn\on{-f.g.}})$.

%\smallskip

%\noindent{\em(c)} The restriction of the functor $J$ to $\Ind^\wedge(\fg\mod_\chi^{\fn\on{-f.g.}})$
%is t-exact. 

\end{thm} 

We will deduce \thmref{t:J self-dual on g} from the following result that will be proved in \secref{sss:self-duality of flags}.

\medskip

Let $\Dmod_\lambda(X)^{\on{ULA}}\subset \Dmod_\lambda(X)$ be the full subcategory of objects that 
are ULA with respect to the projection $X\to N\backslash X$, see \secref{sss:ULA} for what this means.
Consider the corresponding subcategory
$$\Ind^\wedge(\Dmod_\lambda(X)^{\on{ULA}})\subset \Dmod_\lambda(X),$$
see \secref{sss:Ind ULA}.

\medskip

We claim: 

\begin{thm}  \label{t:J self-dual on X}
Consider $G$ acting on the category $\Dmod_\lambda(X)$. \hfill

\smallskip

\noindent{\em(a)} 
There is a canonically defined natural transformation of functors
$$\Av^{N^-}_*\circ \oblv_{N}\circ  \Av^N_!\to 
\Av^{N^-}_!\circ\oblv_{N}\circ  \Av^N_*\simeq J,$$
where the LHS is a partially defined functor. 

\smallskip

\noindent{\em(b)} The functor $\Av^N_!$ is defined and the above natural transformation is an isomorphism, on objects 
from $\Ind^\wedge(\Dmod_\lambda(X)^{\on{ULA}})$. 

\end{thm} 

\sssec{Proof of \thmref{t:J self-dual on g}}

Consider the adjoint functors
\begin{equation} \label{e:loc adj}
\on{Loc}:\fg\mod_\chi\rightleftarrows \Dmod_\lambda(X):\Gamma.
\end{equation}

\medskip

By \propref{p:good Av!}, the functor $\Av^{N^-}_!\circ\oblv_{N}\circ  \Av^N_*$ on $\fg\mod_\chi$
is defined and identifies with
$$\Gamma\circ \Av^{N^-}_!\circ\oblv_{N}\circ  \Av^N_*\circ \Loc.$$

Similarly, the functor $\Av^{N^-}_*\circ \oblv_{N}\circ  \Av^N_!$, viewed as taking values in
$\on{Pro}(\fg\mod_\chi^{N^-})$, identifies with
$$\Gamma\circ \Av^{N^-}_*\circ \oblv_{N}\circ  \Av^N_!\circ \Loc.$$

\medskip

Hence, point (a) of \thmref{t:J self-dual on g} follows from point (a) of \thmref{t:J self-dual on X}.

\medskip

For point (b), we will use \propref{p:good Av!} and the following assertion proved in \secref{ss:Loc and ULA}:

\begin{prop}  \label{p:Loc and ULA}  \hfill

\smallskip

\noindent{\em(a)}
The functor $\on{Loc}$ sends objects in $\Ind^\wedge(\fg\mod_\chi^{\fn\on{-f.g.}})$ to objects in 
$\Ind^\wedge(\Dmod_\lambda(X)^{\on{ULA}})$. 

\smallskip

\noindent{\em(b)}
The functor $\Gamma$ sends 
$$\Dmod_\lambda(X)^{\on{ULA}}\to \fg\mod_\chi^{\fn\on{-f.g.}}$$
and 
$$\Ind^\wedge(\Dmod_\lambda(X)^{\on{ULA}})\to 
\Ind^\wedge(\fg\mod_\chi^{\fn\on{-f.g.}}).$$

%\smallskip

%\noindent{\em(c)}
%Every object in $\Dmod_\lambda(X)$ that is ULA
%with respect to $X\to N\backslash X$ is holonomic. 

\end{prop}

Now, the assertion of \thmref{t:J self-dual on g}(b) follows from that of \thmref{t:J self-dual on X}(b) and \propref{p:good Av!}. 

\ssec{ULA vs finite-generation}  \label{ss:Loc and ULA}

In this subsection we prove \propref{p:Loc and ULA}.

\sssec{}

Let $\fn\otimes \CO_X$ be the Lie algebroid on $X$ corresponding to the $\fn$-action on $X$. Let $'\!\on{D}$ be its universal
enveloping D-algebra. We have a commutative diagram
\begin{equation} \label{e:restr algebroid}
\CD
\Dmod_\lambda(X)  @>>> '\!\Dmod(X) \\
@V{\Gamma(X,-)}VV  @VV{\Gamma(X,-)}V  \\
\fg\mod_\chi  @>>>  \fn\mod,
\endCD
\end{equation} 
where the vertical arrows are taken by taking global sections, and the horizontal arrows are given by restriction.

\medskip

An object $\CM\in \Dmod_\lambda(X)$ is ULA with respect to $X\to N\backslash X$ if and only if $X|_{'\!\on{D}}$
is finitely generated. 

\sssec{}

To prove \propref{p:Loc and ULA}(b) we have to show that the right vertical arrow in \eqref{e:restr algebroid}
sends finitely generated objects to finitely generated objects. 

\medskip

The latter is is enough to do at the associated graded
level, and the assertion follows from the fact that $X$ is proper. 

\sssec{}

Let us prove \propref{p:Loc and ULA}(a). It suffices to show that 
for $\CM\in \fg\mod_\chi^\heartsuit\cap \fg\mod_\chi^{\fn\on{-f.g.}}$, the object 
$\Loc(\CM)|_{'\!\on{D}}$ has finitely generated cohomologies. 
 
\medskip

We first consider the case when $\lambda$ is dominant and \emph{regular}, in which case the functor $\on{Loc}$ is t-exact. 

\medskip

Let 
$$\Loc_\fn:\fn\mod\to '\!\Dmod(X)$$
be the functor left adjoint to $$\Gamma:{}'\!\Dmod(X)\to \fn\mod.$$ From \eqref{e:restr algebroid}, we obtain a natural transformation
$$\Loc_\fn(\CM|_{\fn})\to \Loc(\CM)|_{'\!\on{D}}.$$ 

\medskip

Moreover, for $\CM\in \fg\mod_\chi^\heartsuit$, the above map is \emph{surjective} at the level of $H^0$. Since 
$\Loc(\CM)\in \Dmod_\lambda(X)^\heartsuit$, this proves the required assertion. 

\sssec{}

We now consider the case of a general dominant $\lambda$. Let $\mu$ be a dominant integral weight such that $\lambda'=\lambda+\mu$ is regular.
Let $\chi'$ be the corresponding character of $Z(\fg)$. Consider the translation functor
$$T_{\chi\to \chi'}:\fg\mod_\chi \to \fg\mod_{\chi'},$$
and the commutative diagram
$$
\CD
\Dmod_{\lambda}(X) @>{-\otimes \CO(\mu)}>> \Dmod_{\lambda'}(X) \\
@A{\on{Loc}}AA    @A{\on{Loc}}AA  \\
\fg\mod_\chi  @>{T_{\chi\to \chi'}}>> \fg\mod_{\chi'},
\endCD
$$
see \secref{sss:bad trans}.

\medskip 

It is clear that an object $\CF\in \Dmod_{\lambda}(X)$ is ULA with respect to $X\to N\backslash X$ if and only if 
$\CF\otimes \CO(\mu)\in \Dmod_{\lambda'}(X)$ has the same property. 

\medskip

Furthermore, from the description of the functor $T_{\chi\to \chi'}$ in \corref{c:bad trans} it is clear that it sends objects in $\fg\mod_\chi^{\fn\on{-f.g.}}$
to objects in $\fg\mod_{\chi'}$ whose cohomologies are in $\fg \mod_{\chi'}^{\fn\on{-f.g.}}$. Hence, the assertion of \propref{p:Loc and ULA}(a) for 
$\lambda$ follows from that for $\lambda'$.

\begin{rem}
The above prove shows not only that $\Loc(\CM)\in \Ind^\wedge(\Dmod_\lambda(X)^{\on{ULA}})$, but that it is actually in
$\Ind(\Dmod_\lambda(X)^{\on{ULA}})$. 

\medskip

Indeed, \corref{c:bad trans} implies that $T_{\chi\to \chi'}$ sends objects in $\fg\mod_\chi^{\fn\on{-f.g.}}$ to objects 
that can be expressed as filtered colimits of objects in $\fg\mod_{\chi'}^{\fn\on{-f.g.}}$. 

\end{rem} 

\section{The pseudo-identity functor}  \label{s:Psi}

The goal of this section is to prove \thmref{t:J self-dual on X}. In the process of doing so we will introduce
the \emph{pseudo-identity functor}, which will also be the main character of the sequel to this paper. 

\ssec{The pseudo-identity functor: recollections}   \label{ss:Psi}

\sssec{}

Let $\CY$ be a quasi-compact algebraic stack with an affine diagonal, and let us be given a twisting $\lambda$ on $\CY$. 
Let $-\lambda$ denote the opposite twisting.

\medskip

We identify
$$\Dmod_\lambda(\CY)^\vee\simeq \Dmod_{-\lambda}(\CY)$$
via Verdier duality. This allows to identify 
the category
$$\Dmod_{-\lambda,\lambda}(\CY\times \CY)$$ 
with that of continuous endofunctors on $\Dmod_\lambda(\CY)$. Explicitly,
$$\CQ\in  \Dmod_{-\lambda,\lambda}(\CY\times \CY)\, \mapsto\, F_\CQ, \quad
F_\CQ(\CF):=(p_2)_*(\CQ\sotimes (p_1)^!(\CF)),$$
where $p_1,p_2$ are the two projections $\CY\times \CY\rightrightarrows \CY$, and where for a morphism $f$ we denote by
$f_*$ the \emph{renormalized} direct image functor (see \cite[Sect. 9.3]{DrGa1}).

\medskip

Under this identification, the identity functor corresponds to
$$(\Delta_\CY)_*(\omega_\CY)\in \Dmod_{-\lambda,\lambda}(\CY\times \CY),$$
where we note that the functor
$$\Delta_*:\Dmod(\CY)\to \Dmod_{-\lambda,\lambda}(\CY\times \CY)$$ 
is well-defined because the pullback of the $(-\lambda,\lambda)$-twisting along the diagonal map is canonically trivialized.

\sssec{}

The pseudo-identity functor
$$\on{Ps-Id}_\CY:\Dmod_\lambda(\CY)\to \Dmod_\lambda(\CY)$$
is the functor corresponding to the object
$$(\Delta_\CY)_!(k_\CY)\in \Dmod_{-\lambda,\lambda}(\CY\times \CY),$$
where $k_\CY$ is the ``constant sheaf" on $\CY$, i.e., the D-module Verdier dual to $\omega_\CY$. 

\sssec{}

A stack $\CY$ equipped with a twisting $\lambda$ is said to be \emph{miraculous} if the endofunctor
$\on{Ps-Id}_\CY$ is an equivalence. 

\medskip

The following will be proved in the sequel to this paper (however, we do not use this result here): 

\begin{thm}  \label{t:miraculous}
Suppose that $\CY$ has a finite number of isomorphism classes of $k$-points. Then $\CY$ is miraculous.
\end{thm}

%\medskip

%It follows from the definitions that the functor $\on{Ps-Id}_\CY$ is self-dual. I.e., the functor
%$$(\on{Ps-Id}_\CY)^\vee:(\Dmod_\lambda(\CY))^\vee\to (\Dmod_\lambda(\CY))^\vee$$
%corresponds under the identification
%$$(\Dmod_\lambda(\CY))^\vee\simeq \Dmod_{-\lambda}(\CY)$$
%to the functor 
%$$\on{Ps-Id}_\CY: \Dmod_{-\lambda}(\CY)\to \Dmod_{-\lambda}(\CY),$$
%where we have replaced the initial twisting by its opposite. 

%\sssec{}

%Suppose that $\CY$ is acted on by an algebraic group $H$, so that the given twisting on $\CY$ is
%$H$-equivariant.

%\medskip

%Since the objects involved in the definition of 
%$\on{Ps-Id}_\CY$ are $H$-equivariant, we obtain that the functor $\on{Ps-Id}_\CY$ descends to a functor
%$$\on{Ps-Id}_\CY:\Dmod_\lambda(\CY)^H\to \Dmod_\lambda(\CY)^H.$$

%\medskip

%\noindent{\it Warning:} This should not be confused with the functor $\on{Ps-Id}_{H\backslash \CY}$. 

\ssec{Pseudo-identity, averaging and the ULA property}

\sssec{}  \label{sss:ULA}

Let $f:\CZ\to \CY$ be a smooth morphism between smooth algebraic stacks. Let us recall what it means for
an object $\CF\in \Dmod_\lambda(\CZ)$ to be ULA with respect to $f$, see, e.g. \cite[Sect. 3.4]{Ga1}.

\medskip

The property of being ULA is local in the smooth topology on the source and the target, so we can assume
that $\CZ=Z$ and $\CY=Y$ are schemes. 

\medskip

In this case, the sheaf of rings $\on{D}$ of differential operators
on $Z$ contains a subsheaf of rings, denoted $'\!\on{D}$, consisting of differential operators \emph{vertical}
with respect to $f$ (i.e., these are those differential operators that commute with functions pulled-back from $Y$).
Locally, $'\!\on{D}$ is generated by functions and vertical fields that are parallel to the fibers of $f$. 

\medskip

We shall say that $\CM\in \Dmod(Z)$ is ULA with respect to $f$ if it is finitely generated when considered as a $'\!\on{D}$-module (i.e. it has finitely many non-zero cohomologies, and its cohomologies are locally finitely generated over $'\!\on{D}$).

\begin{rem}
The above definition of 
ULA can be thought of as the $D$-module version of the one given in
 \cite[Sect. 5]{BG} in the setting of \'etale sheaves.
 \end{rem}

\sssec{}  \label{sss:Ind ULA}

Let $\Dmod_\lambda(\CZ)^{\on{ULA}}\subset \Dmod_\lambda(\CZ)$ be the full subcategory that consists of objects that are
ULA with respect to $f$.

\medskip

If $\CZ$ is a scheme, then $\Dmod_\lambda(\CZ)$ has finite cohomological dimension, and we have
$$\Dmod_\lambda(\CZ)^{\on{ULA}}\subset \Dmod_\lambda(\CZ)^c.$$

\medskip

As in \secref{sss:n fin} we define the corresponding functor
$$\Ind(\Dmod_\lambda(\CZ)^{\on{ULA}})\to \Dmod_\lambda(\CZ)$$
and a fully faithful embedding 
$$\Ind^\wedge(\Dmod_\lambda(\CZ)^{\on{ULA}})\hookrightarrow  \Dmod_\lambda(\CZ),$$
whose essential image consists of objects
whose cohomologies are filtered colimits of objects from
$$\Dmod_\lambda(\CZ)^{\heartsuit,\on{ULA}}:=\Dmod_\lambda(\CZ)^\heartsuit\cap \Dmod_\lambda(\CZ)^{\on{ULA}}.$$

\begin{rem}
We note that as in Remark \ref{r:ind n fin} it is not clear to the authors whether, when $\CZ$ is a scheme, the inclusion
$\Ind(\Dmod_\lambda(\CZ)^{\on{ULA}})\subset \Ind^\wedge(\Dmod_\lambda(\CZ)^{\on{ULA}})$
is actually an equivalence. 
\end{rem}

\sssec{}

We take $X$ to be a smooth proper scheme acted on by a group $H$, and we take $\CZ=X$ and $\CY=H\backslash X$.

\medskip

We claim:

\begin{thm} \label{t:pseudo-id and av}
Let $\lambda$ be a $H$-equivariant twisting on $X$. \hfill

\smallskip

\noindent{\em(a)}
There exists a canonically defined natural transformation
\begin{equation} \label{e:pseudo-id and av}
\Av^H_!\to \on{Ps-Id}_{H\backslash X}\circ \Av^H_*[2\dim(X)],
\end{equation} 
(where the left-hand side is a partially defined functor).

\smallskip

\noindent{\em(b)}
The map \eqref{e:pseudo-id and av} is an isomorphism when evaluated on 
objects from $\Ind^\wedge(\Dmod_\lambda(\CZ)^{\on{ULA}})$. 

\end{thm}

\sssec{Proof of \thmref{t:pseudo-id and av}, Step 0}

Consider the Cartesian diagram 
$$
\CD
X  @>{\wt\Delta_{H\backslash X}}>>  X\times H\backslash X  @>{\wt{p}_1}>>  X \\
@V{f}VV   @VV{f\times \on{id}}V   @VV{f}V \\ 
H\backslash X  @>{\Delta_{H\backslash X}}>> H\backslash X\times H\backslash X  @>{p_1}>>  H\backslash X  \\
& & @VV{p_2}V \\
& & H\backslash X. 
\endCD
$$

For $\CF\in \Dmod_\lambda(X)$, the object $\on{Ps-Id}_{H\backslash X}\circ \Av^H_*(\CF)$ identifies with
\begin{multline*}
(p_2)_*\left((\Delta_{H\backslash X})_!(k_{H\backslash X})\sotimes (p_1^!\circ f_*(\CF))\right)\simeq 
(p_2)_*\left((\Delta_{H\backslash X})_!(k_{H\backslash X})\sotimes  ((f\times \on{id})_*\circ \wt{p}_1^!(\CF))\right)\simeq \\
\simeq 
(p_2\circ (f\times \on{id}))_*\left((f\times \on{id})^!\circ (\Delta_{H\backslash X})_!(k_{H\backslash X}) \sotimes \wt{p}_1^!(\CF)\right)\overset{\text{smooth base change}}\simeq \\
\simeq (p_2\circ (f\times \on{id}))_*\left((\wt\Delta_{H\backslash X})_!\circ f^!(k_{H\backslash X})\sotimes \wt{p}_1^!(\CF)\right) \simeq \\
\simeq (\wt{p}_2)_*\left((\wt\Delta_{H\backslash X})_!(k_X)\sotimes \wt{p}_1^!(\CF)\right)[2\dim(H)],
\end{multline*} 
where $\wt{p}_2=p_2\circ (f\times \on{id})$, and we have used the fact that
$$f^!(k_{H\backslash X})\simeq k_X[2\dim(H)].$$

\medskip

Note also that
\begin{multline*}
\Av^H_!(\CF)\simeq f_!(\CF)[2\dim(H)]
=(\wt{p}_2)_!\circ (\wt{\Delta}_{H\backslash X})_!(\CF)[2\dim(H)] =\\
=(\wt{p}_2)_!\circ (\wt{\Delta}_{H\backslash X})_!\circ 
(\wt{\Delta}_{H\backslash X})^*\circ \wt{p}_1^*(\CF)[2\dim(H)]\overset{\text{projection formula}}\simeq \\
\simeq (\wt{p}_2)_!\left( (\wt{\Delta}_{H\backslash X})_!(k_X)\overset{*}\otimes \wt{p}_1^*(\CF)\right)[2\dim(H)].
\end{multline*} 

\sssec{Proof of \thmref{t:pseudo-id and av}, Step 1}

Thus we are reduced to considering the diagram
$$
\CD
X  @>{\wt\Delta_{H\backslash X}}>>  X\times H\backslash X  @>{\wt{p}_1}>>  X \\
& & @V{\wt{p}_2}VV  \\
& & H\backslash X.
\endCD
$$

Since $X$ was assumed proper, the map $\wt{p}_2$ is proper, hence
$(\wt{p}_2)_!\simeq (\wt{p}_2)_*$. The map $\wt{p}_1$ is smooth of relative dimension $\dim(X)-\dim(H)$, so
$$\wt{p}_1^*(\CF)\simeq \wt{p}_1^!(\CF)[-2(\dim(X)-\dim(H))].$$

\medskip

Hence, it suffices to construct a natural transformation
\begin{equation} \label{e:ULA morphism}
(\wt{\Delta}_{H\backslash X})_!(k_X)\overset{*}\otimes \wt{p}_1^!(\CF)\to 
(\wt\Delta_{H\backslash X})_!(k_X)\sotimes \wt{p}_1^!(\CF)[2(2\dim(X)-\dim(H))]
\end{equation} 
and show that it is an isomorphism if $\CF$ is ULA with respect to $f$. 

\sssec{Proof of \thmref{t:pseudo-id and av}, Step 2}

Consider the morphism
$$(f\times \on{id}):X\times H\backslash X \to H\backslash X\times H\backslash X.$$

\medskip

For any $\CF_1\in \Dmod_{-\lambda,\lambda}(H\backslash X\times H\backslash X)$ and
$\CF_2\in \Dmod_{-\lambda,\lambda}(X\times H\backslash X)$ 
we have the canonical map (see \cite[Sect. 2.3]{Ga1}):
$$(f\times \on{id})^*(\CF_1) \overset{*}\otimes \CF_2 \to (f\times \on{id})^*(\CF_1) \sotimes \CF_2[2(2\dim(X)-\dim(H))].$$

This morphism is an isomorphism for $\CF_2$ that is ULA with respect to $f\times \on{id}$. Since all the functors involved are 
continuous, this remains true if $\CF_2$ is a colimit of ULA objects. Further, since the functors
involved have a bounded cohomological amplitude, the same is true if the 
cohomologies of $\CF_2$ have this property.  

\medskip

We take $\CF_2=\wt{p}_1^!(\CF)$. The assumption that $\CF$ has cohomologies that are ULA with respect to $f$
implies that $\CF_2$ has the same property with respect to $f\times \on{id}$. 

\medskip

We take $\CF_1=(\Delta_{H\backslash X})_!(k_{H\backslash X})$. Then 
$$(f\times \on{id})^*(\CF_1)\simeq (\wt\Delta_{H\backslash X})_!(k_X).$$

This yields the desired (iso)morphism \eqref{e:ULA morphism}.

\qed

\begin{rem}

In Theorem \ref{t:pseudo-id and av}, we could have taken $X$ to be a smooth and proper scheme, and 
$f: X \to \CY$ a smooth map (not necessarily a quotient map). Of course, 
$\on{Av}_*^H$ would mean $f_*$, while $\on{Av}_!^H$ would mean the (partially defined) left adjoint of $f^*$, i.e. 
$f_! [2(\dim (X) - \dim (\CY))]$.

\end{rem}

\ssec{A variant}

\sssec{}

Let $H'\subset H$ be a subgroup. Consider the forgetful functor
$$\oblv_{H/H'}:\Dmod_\lambda(H\backslash X)\to \Dmod_\lambda(H'\backslash X).$$
Its right adjoint $\Av^{H/H'}_*$ is given by *-direct image along 
$$H'\backslash X \to H\backslash X,$$
and the partially defined left adjoint $\Av^{H/H'}_!$ is given by !-direct image along the above morphism, shifted by 
$2(\dim(H)-\dim(H'))$. 

\sssec{}

Let $H'_{\on{red}}$ denote the reductive quotient of $H'$. We have:

\begin{thm} \label{t:pseudo-id and av rel} \hfill

\smallskip

\noindent{\em(a)}
There exists a canonically defined natural transformation
\begin{equation} \label{e:pseudo-id and av rel}
\Av^{H/H'}_!\to \on{Ps-Id}_{H\backslash X}\circ \Av^{H/H'}_*[2\dim(X)-\dim(H'_{\on{red}})],
\end{equation} 
(where the left-hand side is a partially defined functor).

\smallskip

\noindent{\em(b)}
The map \eqref{e:pseudo-id and av rel} is an isomorphism when evaluated on 
objects from $\Ind^\wedge(\Dmod_\lambda(H'\backslash X)^{\on{ULA}})$, where the ULA condition 
is taken with respect to the projection $f:H'\backslash X\to H\backslash X$.

\end{thm}

\sssec{}

The proof repeats verbatim that of \thmref{t:pseudo-id and av rel} with the following modification: 

\begin{lem}
Let $X$ be a proper scheme acted on by a group $H'$. Let $p$ denote the projection
$H'\backslash X\to \on{pt}$. Then we have a canonical isomorphism of functors
$$p_*\simeq p_![-2\dim(H')+\dim(H'_{\on{red}})].$$
\end{lem}

\begin{proof}

Factor the map $p$ as
$$H'\backslash X\to H'\backslash \on{pt}\to \on{pt}.$$

The first arrow is proper, and this reduces the assertion of the lemma to the case $X=\on{pt}$.
In the latter case, this is an easy verification.

\end{proof}

\begin{rem}
In the above lemma, it is crucial that we understand $p_*$ as the \emph{renormalized} direct image 
functor of \cite[Sect. 9.3]{DrGa1} (i.e., the continuous extension of the restriction of the usual $p_*$ to
compact objects).
\end{rem}

\ssec{First applications}

In this subsection we will take $G$ to be a reductive group, $X$ its flag variety, and 
$H=N$ the unipotent radical of a Borel. 

\sssec{}

We are going to show:

\begin{thm} \label{t:serre}
The functor $\on{Ps-Id}_{N\backslash X}$ induces a self-equivalence of $\Dmod_\lambda(N\backslash X)$. Moreover, 
$\on{Ps-Id}_{N\backslash X}[2\dim(X)]$ identifies with the composite
$$\Dmod_\lambda(N\backslash X) \overset{\Upsilon^{-1}}\longrightarrow 
\Dmod_\lambda(N^-\backslash X)  \overset{(\Upsilon^-)^{-1}}\longrightarrow \Dmod_\lambda(N\backslash X).$$
\end{thm}

The proof is based on the following assertion, proved below: 

\begin{prop}  \label{p:equiv ULA}
Any object in the essential image of the forgetful functor
$$\oblv_{N^-}:\Dmod_\lambda(N^-\backslash X)^c\to \Dmod_\lambda(X)$$
is ULA with respect to $X\to N\backslash X$.
\end{prop} 

\begin{proof}[Proof of  \thmref{t:serre}] 
 
Since, by \propref{p:Ups X}, the functors $\Upsilon$ and $\Upsilon^-$ are equivalences with
$$(\Upsilon^-)^{-1}\simeq \Av^N_!\circ \oblv_{N^-},$$
it suffices to establish an isomorphism
$$\on{Ps-Id}_{N\backslash X} \circ \Av^N_* \circ \oblv_{N^-}[2\dim(X)]\simeq  \Av^N_! \circ \oblv_{N^-}.$$

However, the latter follows from \propref{p:equiv ULA} and \thmref{t:pseudo-id and av}. 

\end{proof}

\begin{rem}
The fact that $\on{Ps-Id}_{N\backslash X}$ is
an equivalence is also a special case of \thmref{t:miraculous}.
\end{rem} 

\sssec{Proof of \thmref{t:J self-dual on X}}  \label{sss:self-duality of flags}

We start with the (iso)morphism of \thmref{t:pseudo-id and av} and compose it with $\Av^{N^-}_*\circ \oblv_{N}$. We obtain a map
$$\Av^{N^-}_*\circ \oblv_{N}\circ \Av^N_! \to \Av^{N^-}_*\circ \oblv_{N}\circ  \on{Ps-Id}_{N\backslash X}\circ \Av^N_*[2\dim(X)],$$
which is an isomorphism on objects whose cohomologies are ULA with respect to $X\to N\backslash X$.

\medskip

We claim that the RHS, i.e.,  $$\Av^{N^-}_*\circ \oblv_{N}\circ \on{Ps-Id}_{N\backslash X}\circ \Av^N_*[2\dim(X)],$$ is canonically
isomorphic to 
$$\Av^{N^-}_!\circ \oblv_{N}\circ \Av^N_*.$$

In fact, we claim that there is a canonical isomorphism
\begin{equation} \label{e:another relation}
\Av^{N^-}_*\circ \oblv_{N} \circ \on{Ps-Id}_{N\backslash X}[2\dim(X)]\simeq \Av^{N^-}_!\circ \oblv_{N}
\end{equation}
as functors $\Dmod_\lambda(N\backslash X)\to \Dmod_\lambda(N^-\backslash X)$. 

\medskip

Since the functors $\Av^N_*\circ \oblv_{N^-}=\Upsilon$ and $\Av^{N^-}_!\circ \oblv_N$ are mutually inverse, 
the latter isomorphism is equivalent to
$$\Upsilon^- \circ \on{Ps-Id}_{N\backslash X}[2\dim(X)]\simeq \Upsilon^{-1},$$
while the latter is the assertion of \thmref{t:serre}. 

\qed

\sssec{}

As another application of \thmref{t:pseudo-id and av}, we will now prove:

\begin{thm} \label{t:Whit}
let $\CF\in \Dmod(X)$ be $(N^-,\psi)$-equivariant, where $\psi:N^-\to \BG_a$ is a non-degenerate character.
Then there exists a functorial isomorphism (depending on a certain choice)
$$\Av^N_!(\CF)\simeq \Av^N_*(\CF)[2\dim(X)].$$
\end{thm} 

\begin{proof}

By a variant of \propref{p:equiv ULA} (where we replace equivariance by twisted equivariance), we have a canonical isomorphism
$$\on{Ps-Id}_{N\backslash X}\circ \Av^N_*(\CF)[2\dim(X)]\simeq \Av^N_!(\CF).$$

By \thmref{t:serre}, the left-hand side can be further rewritten as
$$\Av^N_!\circ w_0\cdot \Av^N_!\circ w_0\cdot \Av^N_*(\CF),$$
where $w_0\cdot -$ is the functor of translation by (a representative of) the longest element of the 
Weyl group.

\medskip

We claim that there is a functorial isomorphism
$$\Av^N_!\circ w_0\cdot \Av^N_*(\CF)\simeq \Av^N_*(\CF)[\dim(X)],$$
depending on a certain choice. 

\medskip 

Namely, it is known that objects of the form $\Av^N_*(\CF)$ for $\CF\in \Dmod(X)^{N^-,\psi}$
are \emph{canonically} of the form 
$$\CM\underset{\End(\Xi)}\otimes \Xi,\quad \CM\in \End(\Xi)^{\on{op}}\mod,$$
for a choice of the ``big projective" $\Xi\in (\Dmod(X)^N)^\heartsuit$. (Indeed, such objects are right-orthogonal to the other indecomposable projectives in $(\Dmod(X)^N)^\heartsuit$). 

\medskip

Denote 
$$\Xi':=\Av^N_!\circ w_0(\Xi)[-\dim(X)].$$

We obtain 
$$\Av^N_!\circ w_0\cdot \Av^N_*(\CF)\simeq \Av^N_!\circ w_0(\CM\underset{\End(\Xi)}\otimes \Xi)
\simeq \CM\underset{\End(\Xi)}\otimes \Xi'[\dim(X)].$$

\medskip

Now, it is also known that 
$$\Xi':=\Av^N_!\circ w_0(\Xi)[-\dim(X)]$$
is \emph{non-canonically} isomorphic again to $\Xi$, in a way compatible with the action of $\End(\Xi)$.

\medskip

A choice of such an isomorphism gives rise to an identification
$$\CM\underset{\End(\Xi)}\otimes \Xi'[\dim(X)]\simeq \CM\underset{\End(\Xi)}\otimes \Xi[\dim(X)]\simeq
\Av^N_*(\CF),$$ 
as desired. 

\end{proof}

\ssec{Transversality and the proof of \propref{p:equiv ULA}}

\sssec{}

Let $H_1$ and $H_2$ be two groups acting on a smooth variety $X$. We shall say that these two actions are \emph{transversal}
if for every point $x\in X$, the orbits $$H_1\cdot x\subset X\supset H_2\cdot x$$ are transversal at $x$.

\medskip

\begin{lem}  \label{l:trans}
The actions of $H_1$ and $H_2$ on $X$ are transversal if and only if the map
\begin{equation} \label{e:action map}
H_1\times H_2\times X\to X\times X, \quad (h_1,h_2,x)\mapsto (h_1\cdot x,h_2\cdot x)
\end{equation} 
is smooth.
\end{lem}

\sssec{Example} It is easy to see that for $X$ being the flag variety of $G$ and $H_1=N$,
and $H_2=N^-$, then the corresponding actions are transversal.

\sssec{}

We have the following generalization of \propref{p:equiv ULA}:

\begin{prop}  \label{p:trans}
Let the actions of $H_1$ and $H_2$ on $X$ be transversal. Then for any (twisted) D-module $\CF$ on $H_1\backslash X$, if the (twisted) D-module $\oblv_{H_1}(\CF)$ on $X$ is coherent, it is ULA with respect to the projection $X\to H_2\backslash X$.
\end{prop}

\begin{proof}

Consider the Cartesian diagram
$$
\CD
H_1\times H_2\times X @>>>  H_1\times X  @>{\on{act}_1}>>  X \\
@VVV   @V{\on{pr}_1}VV  @V{f_1}VV  \\
H_2\times X   @>{\on{pr}_2}>>  X   @>{f_1}>>  H_1\backslash X \\
@V{\on{act}_2}VV  @VV{f_2}V  \\
X   @>{f_2}>>  H_2\backslash X.
\endCD
$$

\medskip

The property of being ULA is smooth-local with respect to the base, so it is enough to show that the pullback
of $\oblv_{H_1}(\CF)$ to $H_2\times X$ is ULA with respect to the map $\on{act}_2$.

\medskip

The ULA property is also smooth-local with respect to the source. Hence, it suffices to show that the further pullback
of $\CF$ to $H_1\times H_2\times X$ is ULA with respect to the composite left vertical arrow. 

\medskip

However, the latter arrow factors as 
$$H_1\times H_2\times X \to X\times X \overset{p_2}\longrightarrow X\to H_1\backslash X.$$

Since the map $H_1\times H_2\times X \to X\times X$ is smooth, it suffices to show that the pullback of $\CF':=\oblv_{H_1}(\CF)$ 
along $X\times X \overset{p_2}\longrightarrow X$ is ULA with respect to $p_1$. However, this is true for any coherent object $\CF'$. 

\end{proof} 

\section{The case of a symmetric pair}  \label{s:gK}

\ssec{Adjusting the previous framework}

\sssec{}

In this section we will take $G$ equipped with an involution $\theta$; set $K:=G^\theta$.
Let $P$ be a minimal parabolic compatible with $\theta$ (i.e., minimal among parabolics for
which $\theta(P)$ is opposite to $P$); denote $P^-:=\theta(P)$.

\medskip

We change the notations, and in this section denote by $N$ (resp., $N^-$) the unipotent radical of $P$ (resp., $P^-$) and
$$M_K:=P\cap P^- \cap K.$$

\medskip

We have the following basic assertion:

\begin{lem}  \label{l:parab orbits}\hfill
	
\smallskip
	
\noindent{\em(a)} The groups $K$, $M_K \cdot N$ and $M_K \cdot N^-$ act on $X$ with finitely many orbits.
	
\smallskip
	
\noindent{\em(b)}	The actions of $M_K \cdot N$ and $K$ on $X$, as well as the actions of 
$M_K \cdot N$ and $M_K \cdot N^-$ on $X$, are transversal.

\end{lem}

%\begin{cor}  \label{c:mon}
%Every $M_K\cdot N$-equivariant (twisted) D-module on $X$ is $P$-monodromic.
%\end{cor}

\sssec{}  \label{sss:modify}

The discussion in \secref{s:J} needs to be modified as follows: instead of the functor
$$\oblv_N:\CC^N\to \CC$$
and its right and (partially defined) left adjoints $\Av^N_*$ and $\Av^N_!$, we consider the analogous functor 
$$ \oblv_{M_K \cdot N / M_K} : \CC^{M_K \cdot N} \to \CC^{M_K}$$ and its right and (partially defined) left adjoints 
$\Av_*^{M_K \cdot N / M_K }$ and $\Av_!^{M_K \cdot N / M_K}$. 
However, by abuse of notation, we will still write $\oblv_N$ instead of $\oblv_{M_K \cdot N / M_K}$, etc.

\medskip

In \propref{p:main} we take $\Upsilon$ to be the functor 
$$\CC^{M_K\cdot N^-}\to \CC^{M_K\cdot N}$$ given by $\Av_*^{N} \circ \oblv_{N^-}$.

\medskip

The key fact is that this functor is an equivalence for $\CC=\Dmod_\lambda(X)$ (with the same proof), and
hence also for $\fg\mod_\chi$. It's inverse is again given by $ \Upsilon^{-1} = \Av_!^{N^-} \circ \oblv_N$.

\medskip

The assertions of Theorems \ref{t:J self-dual on g} and \ref{t:J self-dual on X} should be modified as follows.
Let $$(\fg\mod_\chi^{M_K})^{\fn\on{-f.g.}}\subset \fg\mod_\chi^{M_K}$$ be the full subcategory equal to the
preimage of $\fg\mod_\chi^{\fn\on{-f.g.}}\subset \fg\mod_\chi$  under the forgetful functor
$\oblv_{M_K}:\fg\mod_\chi^{M_K}\to \fg\mod_\chi$.

\begin{thm} \label{t:J self-dual on g K} \hfill

\smallskip

\noindent{\em(a)} 
There is a canonically defined natural transformation of functors
$$\Av^{N^-}_*\circ \oblv_{N}\circ  \Av^N_!\to 
\Av^{N^-}_!\circ\oblv_{N}\circ  \Av^N_*\simeq J,$$
considered as functors $\fg\mod_\chi^{M_K}\to \fg\mod_\chi^{M_K\cdot N^-}$. 

\smallskip

\noindent{\em(b)} The above natural transformation is an isomorphism when evaluated on objects from 
$\Ind^\wedge((\fg\mod_\chi^{M_K})^{\fn\on{-f.g.}})$.

%\smallskip

%\noindent{\em(c)} The restriction of the functor $J$ to $\Ind^\wedge(\fg\mod_\chi^{\fn\on{-f.g.}})$
%is t-exact. 

\end{thm} 

\begin{thm}  \label{t:J self-dual on X K} \hfill

\smallskip

\noindent{\em(a)} 
There is a canonically defined natural transformation of functors
$$\Av^{N^-}_*\circ \oblv_{N}\circ  \Av^N_!\to 
\Av^{N^-}_!\circ\oblv_{N}\circ  \Av^N_*\simeq J,$$
considered as functors 
$$\Dmod_\lambda(M_K\backslash X)\to \Dmod_\lambda(M_K\cdot N^-\backslash X).$$

\smallskip

\noindent{\em(b)} The above natural transformation is an isomorphism on 
$\Ind^\wedge(\Dmod_\lambda(M_K\backslash X)^{\on{ULA}})$, where the ULA condition 
is taken with respect to the projection $M_K \backslash X\to M_K N\backslash X$. 

\end{thm} 

\sssec{}

\thmref{t:J self-dual on X K} is proved in the same manner as \thmref{t:J self-dual on X}, using \propref{t:pseudo-id and av rel} 
(with $H = M_K \cdot N$ and $H' = M_K$) and the following analog of \thmref{t:serre}:

\begin{thm} \label{t:serre K}
The functor $\on{Ps-Id}_{M_K \cdot N\backslash X}$ is a self-equivalence of $\Dmod_\lambda(M_K \cdot N\backslash X)$. 
Moreover, $\on{Ps-Id}_{M_K \cdot N\backslash X}[2 \dim (X) - \dim (M_K)]$ identifies with the composite
$$\Dmod_\lambda(M_K \cdot N\backslash X) \overset{\Upsilon^{-1}}\longrightarrow 
\Dmod_\lambda(M_K \cdot N^-\backslash X) \overset{(\Upsilon^-)^{-1}}\longrightarrow 
\Dmod_\lambda(M_K \cdot N\backslash X).$$
\end{thm}

\medskip

\thmref{t:J self-dual on g K}(a) is proved as \thmref{t:J self-dual on g}(a).
\thmref{t:J self-dual on g K}(b) is proved using the following version of \propref{p:Loc and ULA}(a):

\begin{prop}  \label{p:Loc and ULA gK} \hfill

\smallskip

\noindent{\em(a)}
The functor $\on{Loc}$ sends 
$$\Ind^\wedge((\fg\mod^{M_K}_\chi)^{\fn\on{-f.g.}})\to 
\Ind^\wedge(\Dmod_\lambda(M_K\backslash X)^{\on{ULA}}).$$

\smallskip

\noindent{\em(b)}
The functor $\Gamma$ sends 
$$\Dmod_\lambda(M_K\backslash X)^{\on{ULA}}\to (\fg\mod^{M_K}_\chi)^{\fn\on{-f.g.}}$$
and 
$$\Ind^\wedge(\Dmod_\lambda(M_K\backslash X)^{\on{ULA}})\to 
\Ind^\wedge((\fg\mod^{M_K}_\chi)^{\fn\on{-f.g.}}).$$

\end{prop}

\begin{proof}

For point (a), it is enough to see that $\on{Loc}$ sends an object $\CF$ in 
$(\fg\mod_\chi^{M_K})^{\fn\on{-f.g.}}$ to an object whose cohomologies are ULA with respect to $M_K \backslash X \to M_K \cdot N \backslash X$. 

\medskip

Being ULA is smooth-local on the source, hence it is enough to see that the cohomologies of $\oblv_{M_K} \circ\on{Loc} (\CF)$ are ULA with 
respect to $X \to M_K\cdot N \backslash X$. 

\medskip

For this, it is enough to see that the cohomologies of $\oblv_{M_K}\circ \on{Loc} (\CF) \simeq \on{Loc} \circ \oblv_{M_K} (\CF)$ are ULA with 
respect to $X \to N \backslash X$, which is the case by \propref{p:Loc and ULA}.

\medskip

Point (b) is proved similarly.

\end{proof}

\ssec{The Casselman-Jacquet functor for $(\fg,K)$-modules}

\sssec{}

In this subsection we will prove:

\begin{thm}  \label{t:J exact on K}  \hfill

\smallskip

\noindent{\em(a)} 
The functor
$$\fg\mod^K_\chi\overset{\oblv_{K/M_K}}\longrightarrow \fg\mod^{M_K}_\chi \overset{J}\to 
\fg\mod_\chi^{M_K\cdot N^-}$$ 
identifies canonically with
$$\Av^{N^-}_*\circ \oblv_{N}\circ  \Av^N_!\circ \oblv_{K/M_K}$$
and is t-exact.

\smallskip

\noindent{\em(b)} The functor 
$$\Dmod_\lambda(X)^K \overset{\oblv_{K/M_K}}\longrightarrow \Dmod_\lambda(X)^{M_K}\overset{J}\to \Dmod_\lambda(X)^{M_K\cdot N^-}$$ 
identifies canonically with
$$\Av^{N^-}_*\circ \oblv_{N}\circ  \Av^N_!\circ \oblv_{K/M_K}$$
and is t-exact.
\end{thm}

\sssec{}

We will first prove:

\begin{prop}  \hfill  \label{p:K impl}

\smallskip

\noindent{\em(a)} The functor $\oblv_{K/M_K}$ maps $(\fg\mod^K_\chi)^c$ to $(\fg\mod_\chi^{M_K})^{\fn\on{-f.g.}}$.

\smallskip

\noindent{\em(b)} The functor $\oblv_{K/M_K}$ maps $\Dmod_\lambda(K\backslash X)^c$ to objects in 
$\Dmod_\lambda(M_K \backslash  X)^{\on{ULA}}$,
where the ULA condition is taken with respect to $M_K \backslash X\to M_K N\backslash X$.

\smallskip

\noindent{\em(c)} The functor $\oblv_K$ maps $\Dmod_\lambda(K\backslash X)^c$ to objects in $\Dmod_\lambda(X)$ that
are holonomic.

\end{prop} 

\begin{proof}

Point (a) is well-known: it is enough to show that objects of the form
$$(U(\fg)\underset{U(\fk)}\otimes \rho)\underset{Z(\fg)}\otimes k,\quad \rho\in \Rep(K)^{\on{f.d.}}$$
where $Z(\fg)\to k$ is given by $\chi$, belong to  $\fg\mod_\chi^{\fn\on{-f.g.}}$.
For that it suffices to show that $U(\fg)\underset{U(\fk)}\otimes \rho$ is finitely generated
over $U(\fn)\otimes Z(\fg)$, and this 
follows from the corresponding assertion at the associated graded level. 

\medskip

To show point (b), since the property of being ULA is smooth-local on the source, it is enough to show that $\oblv_{K}$ maps  
$\Dmod_\lambda(K\backslash X)^c$ to objects in $\Dmod_\lambda(X)$ that are ULA with respect to $X \to M_K \cdot N \backslash X$. 
This follows from \propref{p:trans} and \lemref{l:parab orbits}(b).

\medskip

Alternative proof: change the twisting $\lambda$ by an integral amount to make 
$\Gamma$ an equivalence. Then point (b) follows from point (a), combined with \propref{p:Loc and ULA gK}.

\medskip

Finally, point (c) follows from the fact that the group $K$ has finitely many orbits on $X$.

\end{proof} 

\ssec{Proof of \thmref{t:J exact on K}}

\sssec{}

First, we note that the fact that $J\circ \oblv_{K/M_K}$ is isomorphic to $$\Av^{N^-}_*\circ \oblv_{N}\circ  \Av^N_!\circ \oblv_{K/M_K}$$
follows from \thmref{t:J self-dual on g K}(b) and \propref{p:K impl}(a)
(resp., \thmref{t:J self-dual on X K}(b) and \propref{p:K impl}(b)). 

\medskip

To prove the t-exactness, we proceed as follows: 

\sssec{Step 1}

First, we claim that the functor $J\circ \oblv_{K/M_K}$ is left t-exact for $\Dmod_\lambda(K\backslash X)$.

\medskip

This statement is insensitive to changing $\lambda$ by an integral twisting.  Hence, we can assume that $\lambda$
is such that $\Gamma$ is an equivalence. Since $\Gamma$ is t-exact, the assertion now follows from Propositions \ref{p:K impl}(a)
and \ref{p:J left t exact}.

\sssec{Step 2}

We now claim that the functor $J\circ \oblv_{K/M_K}$ is right t-exact, still for $\Dmod_\lambda(K\backslash X)$.

\medskip

Indeed, this follows by Verdier duality from the previous step, using \propref{p:K impl}(c).

\sssec{Step 3}

It remains to show that $J\circ \oblv_{K/M_K}$ is right t-exact on $\fg\mod_\chi^K$. Since $\lambda$ was chosen so that $\Gamma$ is 
t-exact, the functor $\on{Loc}$ is right t-exact. 

\medskip

We have:
$$J\circ \oblv_{K/M_K}\simeq \Gamma\circ J\circ \oblv_{K/M_K}\circ \on{Loc},$$
where the right-hand side is a composition of t-exact and right t-exact functors. 

\qed[\thmref{t:J self-dual on g}]

\ssec{The ``2nd adjointness" conjecture}  \label{ss:2nd adj}

In the previous subsection we studied the functors
$$J\circ \oblv_{K/M_K}:\fg\mod_\chi^K\to \fg\mod_\chi^{M_K\cdot N^-}$$
and 
$$J\circ \oblv_{K/M_K}:\Dmod_\lambda(K\backslash X)\to \Dmod_\lambda(M_K\cdot N^-\backslash X).$$

In this subsection we will study functors in the opposite direction, namely, $$\Av^{K/M_K}_! \text{ and } \Av^{K/M_K}_*$$ that go from 
$\fg\mod_\chi^{M_K\cdot N}$ (or $\fg\mod_\chi^{M_K\cdot N^-}$) to $\fg\mod_\chi^K$ and from 
$\Dmod_\lambda(M_K\cdot N\backslash X)$ (or $\Dmod_\lambda(M_K\cdot N^-\backslash X)$) to $\Dmod_\lambda(K\backslash X)$,
respectively. 

\sssec{}

First, we note the following consequence of \lemref{l:parab orbits}, \propref{p:trans} and \thmref{t:pseudo-id and av rel}:

\begin{cor} \label{c:Av K}
We have a canonical isomorphism $$\Av^{K/M_K}_!\simeq \on{Ps-Id}_{K\backslash X}\circ \Av^{K/M_K}_*[2\dim(X)-\dim(M_K)]$$ as functors 
$\Dmod_\lambda(M_K\cdot N\backslash X)\to \Dmod_\lambda(K\backslash X)$.
\end{cor} 

Combining with \propref{p:good Av!}, we obtain:

\begin{cor}
The partially defined functor $\Av^{K/M_K}_!$ is defined on the essential image of 
$$\oblv_N:\fg\mod_\chi^{M_K\cdot N}\to \fg\mod_\chi^K.$$
\end{cor} 

\sssec{}

For a group $H$, let us write $$\fl_H = \Lambda^{\dim(H)}(\fh)[\dim H].$$ For a pair of groups $H_1\subset H_2$, set 
$\fl_{H_2/H_1} = \fl_{H_2} \otimes \fl_{H_1}^{-1}$.

\medskip

The same symbols will also stand for the functors of tensoring by those lines. 
Set $\CC=\fg\mod_\chi$ or $\CC=\Dmod_\lambda(X)$. We propose the following conjecture:

\begin{conj} \label{c:main}
There exists a canonical isomorphism $$\Av^{K/M_K}_* \simeq \fl_{K/M_K}^{-1} \circ \Av^{K/M_K}_!\circ \Upsilon$$ as functors 
$$\CC^{M_K\cdot N^-}\to \CC^K.$$
\end{conj} 

In light of \corref{c:Av K}, in the case of $\CC=\Dmod_\lambda(X)$, we can reformulate \conjref{c:main} as follows:

\begin{conj} \label{c:funct eq}
The following diagram of functors commutes:
$$
\CD
\CC^K  @<{\Av^{K/M_K}_*}<< \CC^{M_K\cdot N} \\
@V{\fl_{K/M_K}^{-1} \circ \on{Ps-Id}_{K\backslash X}[2\dim (X) - \dim (M_K)]}VV   @VV{\Upsilon^{-1}}V  \\
\CC^K  @<{\Av^{K/M_K}_*}<< \CC^{M_K\cdot N^-}.
\endCD
$$
\end{conj} 

Note that \conjref{c:funct eq} can be thought of as a sort of functional equation for the functor $\Av^{K/M_K}$, cf. \cite{Ga3}. 

\sssec{}

Note that we have two adjoint pairs of functors
$$\Av^{N^-}_!:\CC^K\rightleftarrows \CC^{M_K\cdot N^-}:\Av^{K/M_K}_*$$
and
$$\Av^{K/M_K}_!:\CC^{M_K\cdot N^-} \rightleftarrows \CC^K:\Av^{N^-}_*.$$

\medskip

We obtain that \conjref{c:main} is equivalent to the following one:

\begin{conj} \label{c:Jacquet and J}
The \emph{right} adjoint functor to 
$$\Av^{K/M_K}_*\circ \oblv_{N^-}:\CC^{M_K\cdot N^-}\to \CC^K$$
is given by $$J \circ \oblv_{K/M_K} \circ \fl_{K/M_K}.$$
\end{conj}

\sssec{}

Combining with \thmref{t:J self-dual on g K}, we further obtain that \conjref{c:Jacquet and J} is equivalent to: 

\begin{conj} \label{c:2nd adj}
The \emph{right} adjoint functor to 
$$\Av^{K/M_K}_*\circ \oblv_{N^-}:\CC^{M_K\cdot N^-}\to \CC^K$$
is given by
$$\Upsilon^-\circ \Av^N_!\circ \oblv_{K/M_K} \circ \fl_{K/M_K}.$$
\end{conj}

\sssec{}

We regard \conjref{c:2nd adj} as an analog of Bernstein's \emph{2nd adjointness theorem} for $\fp$-adic groups. 

\medskip

Recall that the latter says that
in addition to the tautological adjunction (the 1st adjointness)
$$r:\bG\mod \rightleftarrows \bM\mod:i$$
(here we denote by $G$ a $\fp$-adic group, by $M$ its Levi subgroup, by $i$ the normalized parabolic induction functor, and by $r$ the
Jacquet functor), we also have an adjunction
$$i:\bM\mod \rightleftarrows \bG\mod:\ol{r},$$
where $\ol{r}$ is the Jacquet functor with respect to the opposite parabolic. 

\medskip

Here is the table of analogies/points of difference between $\fp$-adic groups and symmetric pairs:

\begin{itemize}

\item The analog of $\bG\mod$ is the category $\fg\mod_\chi^K$;

\item The analog of $\bM\mod$ is \emph{not} $\fm\mod_\chi^{M_K}$, but rather $\fg\mod_\chi^{M_K\cdot N}$ (or $\fg\mod_\chi^{M_K\cdot N^-}$);
note that this category explicitly depends on the choice of the parabolic or its opposite. 

\smallskip

\item The analog of the tautological identification $\bM\mod=\bM\mod$ is the intertwining functor $\Upsilon$;

\smallskip

\item The analog of the induction functor $i$ is $\Av^{K/M_K}_*$; 

\smallskip

\item The analog of the Jacquet functor $r$ (resp., $\ol{r}$) is $\Av^{N}_!$ (resp.,  $\Av^{N^-}_!$).

\end{itemize}

\medskip

With these analogies, \conjref{c:2nd adj} says that the right adjoint to the induction functor $\Av^{K/M_K}_*$ is isomorphic to
the the Jacquet functor $\Av^{N}_!$, up to replacing $N$ by $N^-$,  inserting the intertwining functor, and a twist.

\end{document}